\documentclass[11pt,reqno,a4paper]{amsart}
	\usepackage{amsmath,amssymb,amsthm,framed,cancel,tikz,hyperref}

\newtheorem{theorem}{Theorem}[section]
\newtheorem{lemma}[theorem]{Lemma}

\newtheorem{proposition}[theorem]{Proposition}
\newtheorem{corollary}[theorem]{Corollary}

\theoremstyle{definition}
\newtheorem{definition}[theorem]{Definition}

\newtheorem*{axiom}{Axiom}
\theoremstyle{remark} 
\theoremstyle{remark} 
\theoremstyle{remark}

\theoremstyle{remark}
\newtheorem{remark}[theorem]{Remark}

\numberwithin{equation}{section}

\begin{document}
\title[Few operators on Banach spaces $C_0(L\times L)$]{Few operators on Banach spaces $C_0(L\times L)$}

\author{Leandro Candido}
\address{Universidade Federal de S\~ao Paulo - UNIFESP. Instituto de Ci\^encia e Tecnologia. Departamento de Matem\'atica. S\~ao Jos\'e dos Campos - SP, Brasil}
\email{\texttt{leandro.candido@unifesp.br}}
\thanks{ The author was supported by Funda\c c\~ao de Amparo \`a Pesquisa do Estado de S\~ao Paulo - FAPESP No. 2023/12916-1 }

%    General info
%\date{January 1, 1994 and, in revised form, June 22, 1994.}

%\dedicatory{This paper is dedicated for those who will never have a paper dedicated for.}


\begin{abstract}
Using Ostaszewski's $\clubsuit$-principle, we construct a non-metrizable, locally compact, scattered space $L$ in which the operators on the Banach space $C_0(L \times L)$ exhibit a remarkably simple structure. We provide a detailed analysis and, through a series of decomposition steps, offer an explicit characterization of all operators on $C_0(L \times L)$.
\end{abstract}

\maketitle

\section{Introduction}

A well-known problem in the theory of Banach spaces is the characterization of operators that must necessarily exist on a Banach space $X$. This question, likely originating from the early days of functional analysis, has inspired numerous research directions over the decades; see, for example, \cite{Lindenstrauss}, \cite{Shelah}, \cite{GowersMaurey}, \cite{Mor}, \cite{Koz3}, \cite{Pleb}, \cite{Argyros2011}. It remains a topic of significant interest. This is particularly true for Banach spaces of continuous functions on a compact space $C(K)$, as thoroughly discussed in the comprehensive survey \cite{Koz2}, and is especially relevant for $C(K)$ spaces where $K$ is a scattered Hausdorff compact space. Whenever $K$ is an infinite scattered compactum, $C(K)$ contains a complemented subspace isomorphic to $c_0$, which gives rise to many operators. However, there are examples of such spaces where all operators take the form $T = pI + S$, where $p$ is a real scalar, $I$ denotes the identity operator, and $S$ is an operator with separable range; see \cite{Koz1}, \cite{Koz4}, \cite{Koz5}.

Our goal in this paper is to explore the aforementioned question in the context of Banach spaces of the form $C(K \times K)$, and more specifically, spaces of the form $C_0(L \times L)$, where $L$ denotes a locally compact space. On these spaces, various types of operators naturally arise, in addition to the identity operator $I$ and operators with separable range. For instance, the operator $J: C_0(L \times L) \to C_0(L \times L)$, which transposes the coordinates of a function $f$, i.e., $J(f)(x, y) = f(y, x)$ for all $x, y \in L$. Another class of operators are those we refer to as \emph{diagonal multiplications} (see Definition \ref{Def:DiagonalMulti}). Specifically, given functions $g, h \in C_0(L)$, these operators are defined by the formulas $M_g(f)(x, y) = g(y)f(x, x)$ and $N_h(f)(x, y) = h(x)f(y, y)$. Additionally, recalling that $C_0(L \times L)$ is linearly isometric to the injective tensor product of two copies of $C_0(L)$, i.e., $C_0(L) \hat{\otimes}_{\varepsilon} C_0(L)$, operators induced by the tensor product of operators on $C_0(L)$ naturally arise, such as those we refer to as \emph{matrix operators} (see Definition \ref{Def:Matrizes}). Let $R_1, R_2, R_3, R_4 : C_0(L) \to C_0(L)$ be operators with separable range. We denote the corresponding matrix operator by
\[
\left( \begin{smallmatrix}
 R_1 & R_2 \\
 R_3 & R_4
\end{smallmatrix} \right) = R_1 \otimes I + J \circ (R_2 \otimes I) + J \circ (I \otimes R_3) + I \otimes R_4.
\]

Using Ostaszewski's $\clubsuit$ Principle \cite{Osta} and employing techniques similar to those in \cite{candido2} and \cite{CK}, we construct an example of a locally compact, scattered space $L$ that possesses what we refer to as the \emph{collapsing property} (see Definition \ref{Def:Collapsing}). This property leads to the failure of continuity for many linear mappings, enabling us to classify all operators in the Banach space $C_0(L \times L)$. More specifically, the main result of this paper is stated as follows:

\begin{theorem}\label{main1}
Assuming Ostaszewski's $\clubsuit$ principle, there exists a non-metrizable scattered locally compact Hausdorff space $L$ of height $\omega$ such that, for every operator $T: C_0(L \times L) \to C_0(L \times L)$, there exist unique scalars $p,q \in \mathbb{R}$, unique functions $g, h \in C_0(L)$, unique separable range operators $R_1, R_2, R_3, R_4: C_0(L) \to C_0(L)$, and an unique separable range operator $S: C_0(L \times L) \to C_0(L \times L)$ such that 
\[ T = pI + qJ + M_g + N_h + \begin{pmatrix} R_1 & R_3 \\ R_2 & R_4 \end{pmatrix} + S. \]
\end{theorem}

The paper is organized as follows. In Section \ref{Sec-Prelim}, we establish the basic concepts and terminology used throughout this work. Section \ref{Sec-Ostaszewski} introduces the \emph{collapsing property} and constructs an example of a locally compact scattered topological space that satisfies this property (Theorem \ref{compactclub}). In Section \ref{Sec-Auxiliary}, we examine the implications of the collapsing property on certain functions associated with operators on Banach spaces $C_0(L^n)$, which are crucial for proving our main result. Section \ref{Sec-Operators} details the decomposition of operators on $C_0(L \times L)$ through several steps, culminating in Theorem \ref{main1}. Finally, Section \ref{Sec-Geo} explores the geometry of the spaces $C_0(L \times L)$ as a consequence of the decomposition discussed in the previous section.

%----------------------------------------------------------------------------------------------------------------------------------
%----------------------------------------------------------------------------------------------------------------------------------
%----------------------------------------------------------------------------------------------------------------------------------
%----------------------------------------------------------------------------------------------------------------------------------
%----------------------------------------------------------------------------------------------------------------------------------

\section{Basic Terminology}
\label{Sec-Prelim}

In this section, we establish some notation and basic facts that will be used repeatedly throughout the paper. We begin with ordinal numbers, which will always be considered with their usual order topology. As usual, $\omega$ denotes the first infinite ordinal and $\omega_1$ the first uncountable ordinal.

In this research, unless explicitly stated otherwise, $L$ will always denote a locally compact space for which there exists a finite-to-one continuous surjection $\varphi: L \to \omega_1$. This means a continuous surjective function from $L$ onto $\omega_1$ such that $\varphi^{-1}[\{\alpha\}]$ is finite for every $\alpha < \omega_1$. As in \cite{candido2}, we denote the class of all such spaces by $\mathcal{S}$. For an element $L$ of $\mathcal{S}$, $K = L \cup \{\infty\}$ stands for the Aleksandrov one-point compactification of $L$. We will always consider a member $L$ of $\mathcal{S}$ to be equipped with a finite-to-one function $\varphi:L\to \omega_1$ along with the collection $\{L_\alpha:\alpha<\omega_1\}$, where $L_\alpha=\varphi^{-1}[\alpha]=\{x\in L:\varphi(x)\in \alpha\}$ for each $\alpha<\omega_1$. It is important to note that $\{L_{\alpha}:\alpha<\omega_1\}$ forms a cover of $L$ consisting of countable clopen sets such that $L_{\xi}\subseteq L_{\rho}$ whenever $\xi<\rho$. We should also emphasize that every member of the class $\mathcal{S}$ is a first-countable scattered space \cite[Proposition 3.2]{candido2}, and there will be an ordinal number $\rho$ such that the $\rho$-Cantor-Bendixson derivative $L^{(\rho)}=\emptyset$. We call the least such ordinal the \emph{height} of $L$.

For $0 < n < \omega$, $K^n$ denotes the product of $n$ copies of $K$, equipped with the usual product topology. We define the generalized diagonal of $K^n$ as the set $\varDelta(K^n) = \{(x_1, \ldots, x_n) \in K^n : |\{x_1, \ldots, x_n\}| < n\}$, and the boundary of $K^n$ as the set $\partial(K^n) = \{(x_1, \ldots, x_n) \in K^n : \infty \in \{x_1, \ldots, x_n\}\}$. The support of a point $x=(x_1,\ldots,x_n)\in K^n$ is the set $\mathrm{supp}(x)=\{x_1,\ldots,x_n\}$. We say that points $x,y \in K^n$ are disjoint if $\mathrm{supp}(x)\cap \mathrm{supp}(y)=\emptyset$. We denote by $C(K^n)$ the Banach space of all continuous functions $f: K \to \mathbb{R}$, endowed with the norm $\|f\| = \sup_{x \in K} |f(x)|$. The symbol $C_0(L^n)$ denotes the space of all continuous functions $f: L^n \to \mathbb{R}$ that vanish at infinity. We identify $C_0(L^n)$ with the subspace of $C(K^n)$ consisting of all functions $f$ that vanish on $\partial(K^n)$. The topological dual $C(K^n)^*$ is isometrically identified with the space of all signed Radon measures of finite variation $M(K^n)$, endowed with the variation norm. Similarly, $C_0(L^n)^*$ is naturally identified with the subspace $M_0(K^n) = \{\mu \in M(K^n) : |\mu|(\partial(K^n)) = 0\}$.

Additionally, whenever $V$ denotes a clopen subset of $L^n$, we will identify the space $C_0(V)$ isometrically as a subspace of $C_0(L^n)$, consisting of all continuous functions that vanish outside $V$. In particular, for the collection $\{L_\alpha : \alpha < \omega_1\}$ mentioned above, $C_0(L_\rho)$ denotes the subspace of $C_0(L)$ consisting of functions that vanish outside $L_\rho$. It is straightforward to see that for any clopen subset $V$ of $L^n$, the following decomposition holds:
\[C_0(L^n) = C_0(L^n \setminus V) \oplus C_0(V).\]

In this research, the word \emph{operator} always refers to a bounded linear mapping. For all Banach spaces $X$ and $Y$, we write $X \sim Y$ if $X$ and $Y$ are linearly isomorphic, and $X \cong Y$ if they are isometrically isomorphic. All other standard terminology from Banach space theory and set-theoretic topology will be used as in \cite{Se}.

%--------------------------------------------------------------------------------------------------------
%--------------------------------------------------------------------------------------------------------
%--------------------------------------------------------------------------------------------------------
%--------------------------------------------------------------------------------------------------------
%--------------------------------------------------------------------------------------------------------
%--------------------------------------------------------------------------------------------------------

\section{Collapsing Spaces and the Ostaszewski's Principle}
\label{Sec-Ostaszewski}

The aim of this section is to construct an example of a locally compact space $L$ that is sufficiently exotic so that all operators $T: C_0(L^n) \to C_0(L^n)$ are as simple as possible. To achieve this, we will employ techniques developed in \cite{Koz4}, \cite{CK}, and \cite{candido2}. Following the approach of these works, we will utilize Ostaszewski's $\clubsuit$-principle \cite{Osta}. Denoting $\mathcal{L}(\omega_1)$ as the collection of all countable limit points of $\omega_1$, we recall that $\clubsuit$ consists of the following statement:

\begin{axiom}[$\clubsuit$]
There exists a collection $\{F_\alpha : \alpha \in \mathcal{L}(\omega_1)\}$ satisfying the following conditions:
\begin{itemize}
    \item[(1)] $F_\alpha \subset \alpha$ and $\sup F_\alpha = \alpha$ for each $\alpha \in \mathcal{L}(\omega_1)$.
    \item[(2)] For every uncountable subset $\mathcal{A} \subset \omega_1$, there exists $\alpha \in \mathcal{L}(\omega_1)$ such that $F_\alpha \subset \mathcal{A}$.
\end{itemize}
\end{axiom}

Drawing inspiration from \cite[p. 336]{Koz4}, we present a convenient equivalent version of the $\clubsuit$-principle. For an arbitrary set $L$, we define $\mathcal{P}_0(L) = L$. For each $n < \omega$, we let
\[\mathcal{P}_{n+1}(L) = \left[\bigcup_{k \leq n} \mathcal{P}_k(L)\right]^{<\omega}.\]

Then we define $\mathcal{P}(L)=\bigcup_{n<\omega}\mathcal{P}_n(L)$. The \emph{support} of an element $x\in \mathcal{P}_n(L)$ is defined by
\begin{displaymath}\mathrm{supp}(x)= \left\{
\begin{array}{ll}
\{x\} & \text{ if }n=0;\\
\bigcup_{y\in x}\mathrm{supp}(y) & \text{ if }n>0.
\end{array} \right.
\end{displaymath}

\begin{definition}\label{P(alpha)Convergence}
For every ordinal $\alpha$, we say that a subset $ S \subset \mathcal{P}(\alpha) $ is consecutive if, for every $ x $ and $ y \in S $, either $ \mathrm{supp}(x) < \mathrm{supp}(y) $ or $ \mathrm{supp}(y) < \mathrm{supp}(x) $ whenever $ x \neq y $. Moreover, we say that $ S $ converges to $ \alpha $ if $ S $ is consecutive and $\sup\{y : y \in \bigcup_{x \in S}\mathrm{supp}(x)\} = \alpha$.
\end{definition}

\begin{proposition}[$\clubsuit$]\label{EquivClubsuit}There is a collection $\{G_\alpha:\alpha\in \mathcal{L}(\omega_1)\}$ satisfying the following conditions.
\begin{itemize}
\item[(1)] $G_\alpha \subset \mathcal{P}(\alpha)$ and $G_\alpha$ converges to $\alpha$ for each $\alpha\in \mathcal{L}(\omega_1)$.  
\item[(2)] For every consecutive uncountable subset $\mathcal{A}\subset \mathcal{P}(\omega_1)$ there is $\alpha\in \mathcal{L}(\omega_1)$ such that $G_\alpha \subset \mathcal{A}$.
\end{itemize}
\end{proposition}
\begin{proof}
Let $\varPhi: \omega_1 \to \mathcal{P}(\omega_1)$ be any bijection, and note that the set 
$C_{\varPhi} = \{\alpha < \omega_1 : \varPhi[\alpha] = \mathcal{P}(\alpha)\}$ is uncountable. Indeed, given any $\alpha < \omega_1$, define $\alpha_0 = \alpha$, and for any $n < \omega$, assuming that $\alpha_0, \ldots, \alpha_n$ have already been defined, let $\alpha_{n+1} < \omega_1$ be an ordinal greater than $\sup\{y : y \in \bigcup_{x \in \varPhi[\alpha_n]} \mathrm{supp}(x)\}$ and such that $\mathcal{P}(\alpha_n) \subset \varPhi[\alpha_{n+1}]$. The resulting sequence $\{\alpha_n\}_n$ satisfies both $\mathcal{P}(\alpha_n) \subset \varPhi[\alpha_{n+1}]$ and $\varPhi[\alpha_n] \subset \mathcal{P}(\alpha_{n+1})$ for each $n < \omega$.
Hence, if $\alpha_{\omega}=\sup\{\alpha_n:n<\omega\}$, we have $\alpha\leq \alpha_{\omega}$ and 
\[\varPhi[\alpha_{\omega}]=\bigcup_{n<\omega}\varPhi[\alpha_n]=\bigcup_{n<\omega}\mathcal{P}(\alpha_n)=\mathcal{P}(\alpha_\omega)\]
which shows that $C_{\varPhi}$ is unbounded in $\omega_1$.

Next, let $\{F_\alpha\}_{\alpha \in \mathcal{L}(\omega_1)}$ be the collection obtained from the Axiom $\clubsuit$. We define $\{G_\alpha\}_{\alpha \in \mathcal{L}(\omega_1)}$ as  
\begin{displaymath}G_\alpha= \left\{
\begin{array}{ll}
\varPhi[F_\alpha] & \text{ if }\varPhi[F_\alpha]\subset \mathcal{P}(\alpha)\text{ and converges to }\alpha.\\
F_\alpha & \text{ otherwise}.
\end{array} \right.
\end{displaymath}

Hence, $\{G_\alpha\}_{\alpha \in \mathcal{L}(\omega_1)}$ satisfies condition (1) as required. To verify that it also satisfies condition (2), let $\mathcal{A} \subset \mathcal{P}(\omega_1)$ be an arbitrary uncountable consecutive subset. We fix $Y = \varPhi^{-1}[\mathcal{A}]$ and construct a sequence $\{a_\gamma\}_{\gamma < \omega_1}$ by transfinite recursion as follows. We let $\beta_0 = \min\{y : y \in C_\varPhi\}$ and fix $a_0 = \min\{x \in Y : x \in (\beta_0, \omega_1)\}$. Given $0 < \gamma < \omega_1$, assuming that $a_\xi$ is defined for every $\xi < \gamma$, we let $\beta_\gamma = \min\{\beta \in C_\varPhi : \beta \geq \sup\{a_\xi : \xi < \gamma\}\}$ and define
\[a_\gamma = \min\{x \in Y : x \in (\beta_\gamma, \omega_1) \text{ and } \sup\{y : y \in \mathrm{supp}(\varPhi(x))\} \geq \beta_\gamma\}.\]

Consider the set $X=\{a_\gamma:\gamma<\omega_1\}$ which is clearly uncountable. According to $\clubsuit$ there is $\alpha\in \mathcal{L}(\omega_1)$ such that $F_\alpha\subset X$. Since $\varPhi[F_\alpha]\subset \mathcal{A}$ it follows that $\varPhi[F_\alpha]$ is consecutive. Moreover, for every $x\in F_\alpha$, from the construction of $X$, there is $\beta_x\in C_{\varPhi}$ so that $x<\beta_x<\alpha$. Therefore 
\[\varPhi[F_\alpha]=\bigcup_{x\in F_\alpha}\varPhi(x)\subset \bigcup_{x\in F_\alpha}\varPhi[\beta_x]\subset \bigcup_{x\in F_\alpha}\mathcal{P}(\beta_x)=\mathcal{P}(\alpha).\]

Finally, let $\xi<\alpha$ be arbitrary. From the construction of $X$ there is $x \in F_\alpha$ and $\beta_x\in C_\varPhi$ such that $\xi<\beta_x<x$ and $\sup\{y:y \in\mathrm{supp}(\varPhi(x))\}\geq \beta_x$.
We deduce that $\sup\{y:y \in \bigcup_{x\in \varPhi[F_\alpha]}\mathrm{supp}(x)\}=\alpha$. It follows that $\varPhi[F_\alpha]=G_\alpha$ and this completes the proof.
\end{proof}

\begin{remark}
We observe that the previous proposition in fact provides a statement equivalent to Ostaszewski's principle. Indeed, if $\{G_\alpha : \alpha \in \mathcal{L}(\omega_1)\}$ is a collection as above, it is easy to check that $\{S_\alpha : \alpha \in \mathcal{L}(\omega_1)\}$, where $ S_\alpha = \bigcup_{x \in G_\alpha} \mathrm{supp}(x) $, satisfies the conditions of the Axiom $\clubsuit$.
\end{remark}

Before proceeding, we need to recall some terminology from \cite{candido2}. A collection of ordered $ k $-tuples $\mathcal{G} = \{(G_1^\lambda, \ldots, G_k^\lambda) : \lambda \in \varGamma\}$ is said to be $ k $-separated in $ L $ if there are non-zero natural numbers $ m_1, m_2, \ldots, m_k $ such that $\mathcal{G} \subset [L]^{m_1} \times [L]^{m_2} \times \ldots \times [L]^{m_k}$, and $ G_j^\lambda \cap G_i^\gamma = \emptyset $ whenever $ i \neq j $ or $\lambda \neq \gamma$. 

If $\{I_1, \ldots, I_k\}$ denotes a partition of $\{1, 2, \ldots, m\}$, a collection $\{(x_1^\lambda, \ldots, x_m^\lambda) : \lambda \in \varGamma\} \subset L^m$ is said to be $(I_1, \ldots, I_k)$-separated if the family $\{(G_1^\lambda, \ldots, G_k^\lambda) : \lambda \in \varGamma\}$, where $ G_j^\lambda = \{x_i^\lambda : i \in I_j\} $ for each $ j = 1, 2, \ldots, k $, is $ k $-separated in $ L $.

\begin{definition}[Collapsing Property]\label{Def:Collapsing}
For every integer \(1 \leq m < \omega\), let \(\{I_1, \ldots, I_k\}\) be a partition of \(\{1, 2, \ldots, m\}\). We say that an \((I_1, \ldots, I_k)\)-separated collection \(\{(x_1^\lambda, x_2^\lambda, \ldots, x_m^\lambda) : \lambda \in \varGamma\} \subset L^m\) collapses if it has an accumulation point \((q_1, q_2, \ldots, q_m) \in K^m\) such that, for some distinct points \(a_1, a_2, \ldots, a_k\) in \(K\) with \(a_k = \infty\), we have \(\{q_i : i \in I_j\} = \{a_j\}\) for every \(1 \leq j \leq k\).
\end{definition}

\begin{definition}\label{collapsing}
A locally compact Hausdorff space $L$ is said have the \emph{collapsing property} if whenever $\{I_1,\ldots,I_k\}$ is a partition of $\{1,2,\ldots,m\}$, every uncountable $(I_1,\ldots,I_k)$-separated subset of $L^m$ collapses.
\end{definition}

In what follows, our objective is to provide an example of a locally compact space within the class $\mathcal{S}$ possessing the collapsing property. The construction is inspired by \cite[Example 2.17]{DowJunPel} and shares similarities with the steps seen in \cite[Theorem 4.2]{candido2}. Additional related constructions can be found in \cite[p. 337]{Koz4} and \cite[Theorem 3.1]{CK}.

For our construction, it is important to recall Kuratowski's definition of an ordered pair, namely $(x, y) = \{\{x\}, \{x, y\}\}$. According to our terminology, for every ordinal $\alpha$, we have $\alpha \times \alpha \subset \mathcal{P}_2(\alpha)$. More generally, for any ordinal $\alpha$ and non-zero natural numbers $m_1, m_2, \ldots, m_k$, any $k$-tuple in $[\alpha \times \alpha]^{m_1} \times \ldots \times [\alpha \times \alpha]^{m_k}$ can be identified with an element of $\mathcal{P}(\alpha)$ within our terminology.

In the following proof, we denote by $\pi : \omega_1 \times \omega_1 \to \omega_1$ the canonical projection onto the first coordinate, that is, $\pi(x, y) = x$ for every $(x, y) \in \omega_1 \times \omega_1$.

\begin{theorem}[$\clubsuit$]\label{compactclub} 
There exists a locally compact Hausdorff space $L$ in the class $\mathcal{S}$ with height $\omega$ and having the collapsing property.
\end{theorem}
\begin{proof}
We assume that $\clubsuit$ holds and fix a collection $\mathcal{G}=\{G_\alpha:\alpha\in \varGamma\}$ satisfying the conditions of Proposition \ref{EquivClubsuit}. Taking for material points from the set $\omega_1\times \omega_1$, we construct by transfinite induction a family of topological spaces $\{(L_\alpha,\tau_\alpha):\alpha<\omega_1\}$ as follows. We start by fixing $L_0=\{(0,0)\}$ and $\tau_{0}=\{\emptyset,\{(0,0)\}\}$. For an arbitrary $\alpha<\omega_1$ we assume a collection of locally compact Hausdorff spaces $\{(L_\beta,\tau_\beta):\beta<\alpha\}$ satisfying the following conditions for every $\gamma\leq\beta <\alpha$:
\begin{itemize}
\item[(a)] $L_\beta^{(\omega)}=\emptyset$.
\item[(b)]  $\pi^{-1}[\beta]=L_\beta$ and $\pi^{-1}[\{\gamma\}]$ is finite. 
\item[(c)] $L_\gamma$ is a clopen subset of $L_{\beta}$.
\end{itemize}
Let $ L_\alpha^* = \bigcup_{\beta < \alpha} L_{\beta} $ be equipped with the topology $ \tau_{\alpha}^* $ generated by $ \bigcup_{\beta < \alpha} \tau_{\beta} $. It is readily seen that $ (L_\alpha^*, \tau^*_{\alpha}) $ is a locally compact Hausdorff space. Moreover, for condition (a), it follows that $ (L_\alpha^*)^{(\omega)} = \emptyset $. We construct the space $ (L_\alpha, \tau_\alpha) $ in the following manner.

If $\alpha$ is a successor ordinal, then we define $L_\alpha = L_\alpha^* \cup \{(\alpha, \alpha)\}$ and take $\tau_{\alpha}$ as the topology generated by $\tau^*_{\alpha} \cup \{\{(\alpha, \alpha)\}\}$.

If $ \alpha \in \mathcal{L}(\omega_1) $ and $ G_\alpha $ is not a $ k_\alpha $-separated family in $ L_\alpha^* $, denoted as $ \{(G_1^n, \ldots, G_{k_\alpha}^n) : n < \omega\} $ for some $ 1 < k_\alpha < \omega $, or if $ G_\alpha $ has this form but for every $ r < \omega $, there is $ n < \omega $ such that $ (G_1^n \cup \ldots \cup G_k^n) \cap (L_\alpha^*)^{(r)} \neq \emptyset $, then again we set $ L_\alpha = L_\alpha^* \cup \{(\alpha, \alpha)\} $ and let $ \tau_\alpha $ be generated by $ \tau^*_{\alpha} \cup \{\{(\alpha, \alpha)\}\} $.

For otherwise, since $G_\alpha=\{(G_1^n, \ldots, G_{k_\alpha}^n): n < \omega\}$ converges to $\alpha$ in the sense of Definition \ref{P(alpha)Convergence}, there is a strictly increasing sequence $\{\beta_n\}_n$ in $\alpha$ such that $\sup\{\beta_n:n<\omega\}=\alpha$ and $\mathrm{supp}(G_1^n\cup \ldots \cup G_{k_\alpha}^n)\subset [\beta_n+1,\beta_{n+1}]$ for each $n<\omega$.

From condition (b), we have that $G_1^n\cup \ldots \cup G_{k_\alpha}^n\subset L_{\beta_{n+1}}\setminus L_{\beta_n}$ for each $n<\omega$. Furthermore, since $G^n_i\cap G^n_j=\emptyset$ whenever $i\neq j$ for every $n<\omega$, from condition (c), we can fix pairwise disjoint clopen subsets $W^n_1,\ldots,W^n_{k_\alpha}$ contained in $L_{\beta_{n+1}}\setminus L_{\beta_n}$, such that $G^n_j\subset W^n_j$ for each $1\leq j \leq k_\alpha$.

Recalling that, for some $r<\omega$, $(G_1^n \cup \ldots \cup G_k^n) \cap (L_\alpha^*)^{(r)}=\emptyset$, we may assume that $W^n_1\cup \ldots\cup W^n_{k_\alpha}\subset L_\alpha^*\setminus (L_\alpha^*)^{(r)}$ for every $n<\omega$.

Next, for each $n<\omega$ and $0\leq j \leq k_\alpha-2$, we fix
\begin{align*}
V_{n}((\alpha,\alpha+j))=\{(\alpha,\alpha+j)\}\cup\left(\bigcup_{m\geq n} W_{j+1}^{m}\right).
\end{align*}
\normalsize
Then we let $ L_{\alpha} = L_{\alpha}^* \cup \{(\alpha,\alpha), (\alpha,\alpha+1), \ldots, (\alpha,\alpha+k_\alpha-2)\} $ and $ \tau_{\alpha} $ be the topology generated by $ \tau^*_{\alpha} \cup \{V_n((\alpha,\alpha+j)) : n < \omega, \ 0 \leq j \leq k_\alpha-2\} $. We can readily verify that all conditions (a), (b), and (c) are satisfied by the space $ (L_{\alpha}, \tau_{\alpha}) $. With this, we have completed the construction of the collection $ \{(L_\alpha, \tau_\alpha) : \alpha < \omega_1\} $.

We define $ L $ as the union of $ L_\alpha $ for $ \alpha < \omega_1 $, equipped with the topology generated by the union of $ \tau_{\alpha} $ for $ \alpha < \omega_1 $. Consequently, $ L $ becomes a locally compact Hausdorff space.

To demonstrate that $ L $ belongs to the class $ \mathcal{S} $, we let $ \varphi: L \rightarrow \omega_1 $ be the restriction to $ L $ of the canonical projection $ \pi: \omega_1 \times \omega_1 \to \omega_1 $. Condition (b) guarantees that $ \varphi $ is a finite-to-one surjection. Furthermore, it is continuous, as $ \varphi^{-1}[[\beta+1,\alpha]] = L_\alpha \setminus L_\beta $ is an open set whenever $ \beta \leq \alpha < \omega_1 $ according to condition (c).

Let $K=L\cup \{\infty\}$ be the one-point compactification of $L$. To prove that $L$ has the collapsing property, we take an arbitrary $1\leq m<\omega$, a partition $\{I_1,\ldots,I_k\}$ of $\{1,\ldots,m\}$, and let $\mathcal{C}=\{(x_1^\lambda,x_2^\lambda,\ldots,x_m^\lambda):\lambda<\omega_1\}\subset L^m$ be an uncountable $(I_1,\ldots,I_k)$-separated subset.

For each $\lambda<\omega_1$ and $1\leq j \leq k$ we define $A_j^\lambda=\{x_i^\lambda:i\in I_j\}$ and form the the collection $\mathcal{A}=\{(A_1^{\lambda},\ldots, A_{k}^{\lambda}):\lambda<\omega_1\}$ which is clearly $k$-separated in $L$. 

By passing to an uncountable subset of $\mathcal{A}$ if necessary, we may assume that $\mathcal{A}$ is consecutive and there is $r<\omega$ such that $(A_1^{\lambda}\cup \ldots \cup A_{k}^{\lambda})\cap L^{(r)}=\emptyset$ for all $\lambda<\omega_1$. Recalling the properties of our collection $\mathcal{G}$, there is  $\alpha\in \mathcal{L}(\omega_1)$ such that $G_{\alpha}=\{(G_1^{n},\ldots,G_k^{n}):n<\omega_1\}\subset \mathcal{A}$. If $k=1$ we fix $q=(\infty,\stackrel{m}{\ldots},\infty)$. Since $(\alpha,\alpha)$ is an isolated point of $L$, for every compact subset $Z\subset L$, the set $\{n<\omega:Z\cap G_1^n\neq \emptyset\}$ is finite. It follows that, if $\{y_n\}_n$ is sequence such that $y_{n}\in G_1^n$ for all $n<\omega$, then $\{y_n\}_n$ converges to $\infty$. We deduce that $q$ is an accumulation point of $\mathcal{C}$ which shows that $\mathcal{C}$ collapses.

If $k>1$, we let $q=(q_1,\ldots,q_m)$ be such that $q_i=(\alpha,\alpha+j)$ for each $i \in I_{j+1}$ for all $0\leq j<k-2$, and $q_i=\infty$ for each $i \in I_{k}$.  According to the construction of $L$, every sequence $\{y_n\}_n$ in $L$ such that $y_{n}\in G_{j+1}^n$, for every $n<\omega$, converges to $(\alpha,\alpha+j)$ if $0\leq j<k-2$. Furthermore, since $V_{m}((\alpha,\alpha+j))\cap  G^n_{k}=\emptyset$ for every $m,n<\omega$ and $0\leq j<k-2$, it follows that $\{n<\omega:Z\cap G^n_{k}\neq \emptyset\}$ is finite for every compact subset $Z\subset L$. Consequently, every sequence $\{y_n\}_n$ such that $y_{n}\in G^n_{k}$, for all $n<\omega$, converges to $\infty$. We deduce that $q$ is an accumulation point of $\mathcal{C}$ in $K^m$. Therefore, $\mathcal{C}$ collapses. 
 
We prove next that $L$ has height $\omega$. It follows from our construction that every point $x\in L$ admits a neighborhood $V$ such that $V\cap L^{(r)}=\emptyset$ for some $r<\omega$. Hence, $L^{(\omega)}=\emptyset$.
If $L^{(r+1)}=\emptyset$ for some $r<\omega$, as stated in \cite[Proposition 3.4]{candido2}, we infer that $L^{(r)}$ must be uncountable. Consequently, any uncountable $(\{1\},\{2\})$-separated family contained in $L^{(r)}\times L^{(r)}$ can only have $(\infty,\infty)$ as a possible accumulation point in $K^2$. However, this contradicts the collapsing property of $L$. Thus, the proof is concluded.
\end{proof}

\begin{remark}\label{Rem:Singleton}
It is important to observe that, if $L$ denotes the space constructed in Theorem \ref{compactclub}, then whenever $\mathcal{A}$ is an uncountable collection of pairwise disjoint points in $L^n \setminus \varDelta(K^n)$, for $n \geq 1$, $\mathcal{A}$ admits an accumulation point in $L^n \cap \varDelta(K^n)$. Indeed, in this case, we can construct a collection 
$\mathcal{C} = \{(x^\alpha_1, \ldots, x^\alpha_n, y^\alpha)\} \subset L^{n+1} \setminus \varDelta(K^{n+1})$, consisting of pairwise disjoint points such that 
$x_\alpha = (x^\alpha_1, \ldots, x^\alpha_n) \in \mathcal{A}$ for each $\alpha < \omega_1$. It is easily seen that $\mathcal{C}$ is a $(\{1, \ldots, n\}, \{n+1\})$-separated subset of $L^{n+1}$ and thus collapses by the property of $L$, meaning it admits an accumulation point of the form $(a, \ldots, a, \infty)$, where $a \in L$. It follows that $(a, \ldots, a)$ is an accumulation point of $\mathcal{A}$.
\end{remark}

%-------------------------------------------------------------------------------
%-------------------------------------------------------------------------------
%-------------------------------------------------------------------------------
%-------------------------------------------------------------------------------
%-------------------------------------------------------------------------------
%-------------------------------------------------------------------------------

\section{Auxiliary results on weak$^*$ continuous functions}
\label{Sec-Auxiliary}

In this section, we will always consider $L$ to be a locally compact space in the class $\mathcal{S}$ possessing the collapsing property, with $K = L \cup \{\infty\}$ as its Alexandroff one-point compactification. As demonstrated in the previous section, such spaces exist under the assumption of Ostaszewski's principle (Theorem \ref{compactclub}). We assume that $L$ is equipped with a finite-to-one continuous surjection $\varphi: L \to \omega_1$, which induces the clopen cover $\{L_\alpha:\alpha < \omega_1\}$ as described in Section \ref{Sec-Prelim}. 

We will explore the consequences of the collapsing property on continuous functions $\nu: K^n \to M_0(K^m)$, $\nu(x)=\nu_x$, where $M_0(K^m)$ is endowed with the weak$^*$ topology induced by $C_0(L)$, and such that $\nu_x = 0$ for every $x \in \partial K^n$. These functions will henceforth be referred to as weak$^*$ continuous functions that vanish at infinity. The main reason these functions are so relevant to our investigation is due to the following characterization, which holds for all locally compact spaces, see \cite[Theorem VI 7.1]{DunfordSchwartz}.

\begin{theorem}\label{Thm:OperatorCharacterization}
For all $m, n \geq 1$, and for every weak$^*$ continuous function $\nu: K^n \to M_0(K^m)$ that vanishes at infinity, there exists a unique operator $T: C_0(L^m) \to C_0(L^n)$ such that $\nu_x = T^*(\delta_x)$ for each $x \in K^n$. Moreover, the norm of $T$ is given by $\|T\| = \sup_{x \in K^n} \|\nu_x\|$. Conversely, if $T: C_0(L^m) \to C_0(L^n)$ is a bounded linear mapping, then the formula $x \mapsto T^*(\delta_x)$ defines a weak$^*$ continuous function from $K^n$ to $M_0(K^m)$ that vanishes at infinity.
\end{theorem}
\begin{proof}
Given a function $\nu: K^n \to M_0(K^m)$ as stated in the theorem, we define an operator $T: C_0(L^m) \to C_0(L^n)$ by the formula
\[T(f)(x) = \int f \, d\nu_x.\]
It is straightforward to verify that $T$ satisfies the required conditions. The converse statement is evident.
\end{proof}

\begin{remark}\label{Rem:SeAnulaNoInfinito}
It follows from the previous theorem, and the fact that $K$ denotes a scattered compact space, that if $\nu: K^n \to M_0(K^m)$ is weak$^*$ continuous and vanishes at infinity, then for every $f \in C_0(L^m)$, there are at most countably many points $x \in K^n$ such that $\int f \, d\nu_x \neq 0$; see \cite[Corollary 8.5.6]{Se}. Moreover, using an argument similar to the one in \cite[Lemma 4.4]{candido2}, it follows that for any pair of natural numbers $n, m \geq 1$, if $\nu: K^n \to M_0(K^m)$ is weak$^*$ continuous and vanishes at infinity, then for any countable set $F \subset L^m$, the set $\{y \in K^n : |\nu_y|(F) \neq 0\}$ is also countable.
\end{remark}

The following theorems presented in this section are central to our study and constitute the core of this paper. They play a pivotal role in demonstrating the failure of continuity of many linear operators on $C_0(L^n)$.

\begin{theorem}\label{Thm:vanish-B}
Let $n \geq 1$ be an integer, and let $\nu: K^n \to M_0(K^n)$ be a weak$^*$ continuous function that vanishes at infinity. Suppose $A = \{y_{\alpha} : \alpha < \omega_1\}$ and $B = \{z_{\alpha} : \alpha < \omega_1\}$ are subsets of $L^n$ such that $\{\mathrm{supp}(y_\alpha) : \alpha < \omega_1\}$ forms an uncountable $\varDelta$-system with root $R_A$, and $\mathrm{supp}(y_\alpha) \setminus (\mathrm{supp}(z_\alpha) \cup R_A) \neq \emptyset$ for all $\alpha < \omega_1$. Then, there exists $\rho < \omega_1$ such that $\nu_{y_{\alpha}}(\{z_{\alpha}\}) = 0$ whenever $\rho < \alpha < \omega_1$.
\end{theorem}

\begin{proof}
Towards a contradiction, suppose that for every $\beta < \omega_1$, there exists some $\beta < \alpha < \omega_1$ such that $\nu_{y_{\alpha}}(\{z_{\alpha}\}) \neq 0$. By passing to an appropriate uncountable subset of indices if necessary, we may assume there exists $\epsilon > 0$ such that $|\nu_{y_{\alpha}}|(\{z_{\alpha}\})\geq \epsilon$ for all $\alpha < \omega_1$. Furthermore, recalling Remark \ref{Rem:SeAnulaNoInfinito}, we may regard the family $\{\mathrm{supp}(z_\alpha) : \alpha < \omega_1\}$ as an uncountable $\varDelta$-system with root $R_B$.

Moreover, we may assume that there exists a partition $\{I_1, \ldots, I_{k+1}\}$ of $\{1,\ldots,n\}$ such that, for each $\alpha < \omega_1$, if $y_\alpha = (y^\alpha_1, \ldots, y^\alpha_n)$, then $R_A = \{y^\alpha_j : j \in I_{k+1}\}$. For each $i, j \in \{1, \ldots, n\} \setminus I_{k+1}$, we have $y^\alpha_i = y^\alpha_j$ if and only if $i$ and $j$ belong to the same set $I_r$ for some $1 \leq r \leq k$. Additionally, we assume there exists a sequence $\{d_j\}_{j \in I_{k+1}}$ such that $y^\alpha_j = d_j$ for all $j \in I_{k+1}$ and for every $\alpha < \omega_1$.

Similarly, we suppose a partition $\{F_1, \ldots, F_{l+1}\}$ of $\{1, \ldots, n\}$ such that, if $z_\alpha = (z^\alpha_1, \ldots, z^\alpha_n)$, then $R_B = \{z^\alpha_j : j \in F_{l+1}\}$. For each $i, j \in \{1,\ldots,n\} \setminus F_{l+1}$, we have $z^\alpha_i = z^\alpha_j$ if and only if $i, j$ belong to the same set $F_s$ for some $s \in \{1, \ldots, l\}$. We also assume that there exists a sequence $\{c_j\}_{j \in F_{l+1}}$ such that $z^\alpha_j = c_j$ for each $j \in F_{l+1}$ and for all $\alpha < \omega_1$.

If $R_A$ or $R_B$ is empty, we set $I_{k+1} = \emptyset$ or $F_{l+1} = \emptyset$, respectively, and proceed with the proof with minor adjustments.

Since $\mathrm{supp}(y_\alpha)\setminus (\mathrm{supp}(z_\alpha)\cup R_A)\neq \emptyset$ for every $\alpha<\omega_1$, by relabeling the partitions, we may assume that there exists $1 < r \leq k$ such that for each $\alpha < \omega_1$,
\[\{y^\alpha_1, \ldots, y^\alpha_n\} \setminus (\{z^\alpha_1, \ldots, z^\alpha_n\}\cup R_A) = \{y^\alpha_j : j \in I_1 \cup \ldots \cup I_r\},\]
and $\{y^\alpha_j : j \in I_i\} = \{z^\alpha_j : j \in F_{i-r}\}$ for each $r + 1 \leq i \leq k$. 

Since Radon measures on scattered spaces are atomic, for each $\alpha < \omega_1$, there exists a finite set $G_{\alpha} \subset L$ such that $\mathrm{supp}(y_\alpha) \cup \mathrm{supp}(z_\alpha) \subset G_{\alpha}$ and 
\[|\nu_{y_\alpha}|(K^n \setminus G_{\alpha}^n) < \frac{\epsilon}{2}.\]

By applying the $\varDelta$-System Lemma, we may assume that the family $\{G_\alpha : \alpha < \omega_1\}$ forms a $\varDelta$-system with root $Q$, where $(R_A \cup R_B) \subset Q$, and that all the elements of the family have the same cardinality. For each $\alpha$, define $G'_\alpha = G_\alpha \setminus Q$ and write $G'_{\alpha} = \{x^{\alpha}_1, \ldots, x^{\alpha}_m\}$, where for each $1 \leq i \leq r$, $x_i^\alpha = y^\alpha_j$ if and only if $j \in I_i$, and for $r+1 \leq i \leq r+l$, $x_i^\alpha = z^\alpha_j$ if and only if $j \in F_{i-r}$.

Now consider the set $\mathcal{C} = \{(x_1^\alpha, \ldots, x_m^\alpha) : \alpha < \omega_1\}$. Observe that $\mathcal{C}$ can be regarded as a $(\{r+1\},\ldots,\{r+l\}, \{1,\ldots,r,r+l+1, \ldots, m\})$-separated subset of $L^m$. By removing countably many elements from $\mathcal{C}$, we may assume that there exists $\xi < \omega_1$ such that $Q^n \subset L_{\xi}^n$ and $\mathcal{C} \subset (L \setminus L_\xi)^n$. 

Since $L$ has the collapsing property, the set $\mathcal{C}$ has an accumulation point $a = (a_1, \ldots, a_m)$ such that $a_{r+1}, \ldots, a_{r+l}$ are pairwise distinct points of $L \setminus Q$, and $a_i = \infty$ otherwise. Finally, we fix a net $\{(x^{\alpha_{\gamma}}_1, \ldots, x^{\alpha_{\gamma}}_m)\}_{\gamma \in \varGamma}$ in $\mathcal{C}$ that converges to $a$.

Let $U_1, U_2, \ldots, U_{l+1}$ be pairwise disjoint clopen compact subsets of $L$ such that $a_{r+i} \in U_i$ for every $1 \leq i \leq l$, and $R_B \subset U_{l+1}$. Additionally, we consider that $(U_1 \cup \ldots \cup U_{l+1}) \cap (Q\setminus R_B) = \emptyset$. By passing to a subnet if necessary, we may suppose that $U_i \cap G_{\alpha_{\gamma}} = \{x^{\alpha_{\gamma}}_{r+i}\}$ for all $1 \leq i \leq l$, and $U_{l+1} \cap G_{\alpha_{\gamma}} = R_B$ for all $\gamma \in \Gamma$. 

For each $1 \leq j \leq n$, set $V_j = U_{r+i}$ and $v_j = a_{r+i}$ whenever $j \in F_i$ for some $1 \leq i \leq l$. Recalling that $R_B = \{c_j : j \in F_{l+1}\}$, we define $V_j = U_{l+1}$ and $v_j = c_j$ for each $j \in F_{l+1}$. Let $v = (v_1, \ldots, v_n)$, and define the set $P = \{(w_1, \ldots, w_n) : w_j = v_j \text{ if } j \notin F_{l+1}, \text{ and } w_j \in R_B \text{ otherwise}\}$. We choose a clopen neighborhood $V \subset V_1 \times \ldots \times V_n$ of $v$ such that $V \cap P = \{v\}$.

Let $\{v_\gamma\}_{\gamma \in \varGamma}$ be a net where, for each $\gamma$, $v_\gamma = (v^\gamma_1, \ldots, v^\gamma_n)$, and $v^\gamma_j = x^\gamma_{r+i}$ whenever $j \in F_i$ for some $1 \leq i \leq l$, with $v^\gamma_j = c_j$ for each $j \in F_{l+1}$. Since $\{v_\gamma\}_{\gamma \in \varGamma}$ converges to $v$, by the construction of $V$, we may pass to a subnet and assume that $G_{\alpha_\gamma}^n \cap V = \{v_\gamma\}$ for each $\gamma$.

Recalling that $R_A = \{d_j : j \in I_{k+1}\}$, for each $1 \leq j \leq n$, fix $u_j = a_i$ whenever $j \in I_i$ for some $1 \leq i \leq k$, and set $u_j = d_j$ for all $j \in I_{k+1}$. We then define $u = (u_1, \ldots, u_n)$.

Next, let $\{u_{\gamma}\}_{\gamma \in \varGamma}$ be a net where, for each $\gamma$, $u_{\gamma} = (u^{\gamma}_1, \ldots, u^{\gamma}_n)$, and $u^{\gamma}_j = x^{\gamma}_i$ whenever $j \in I_i$ for some $1 \leq i \leq k$, with $u^{\gamma}_j = d_j$ for all $j \in I_{k+1}$. Clearly, $\{u_{\gamma}\}_{\gamma \in \varGamma}$ converges to $u$.

Since for every $\gamma \in \varGamma$ we have $u_{\gamma} \in A$, $v_{\gamma} \in B$, and $G_{\alpha_\gamma}^n \cap V = \{v_\gamma\}$, it follows that
\[|\nu_{u_{\gamma}}(V)| \geq |\nu_{u_{\gamma}}(\{v_{\gamma}\})| - |\nu_{u_{\gamma}}(V \setminus \{v_{\gamma}\})| \geq |\nu_{u_{\gamma}}|(\{v_{\gamma}\}) - |\nu_{u_{\gamma}}|(K^n \setminus G_{\alpha_{\gamma}}^n) \geq \epsilon - \frac{\epsilon}{2} = \frac{\epsilon}{2}.\]
Hence, recalling that $\nu$ is weak$^*$-continuous and that $V$ is a clopen, compact subset of $L^n$, we obtain
\[
|\nu_{u}(V)| = \lim_{\gamma \to \infty} |\nu_{u_{\gamma}}(V)| \geq \frac{\epsilon}{2}.
\]
However, since $u$ is an element of $\partial(K^n)$, we have $\nu_{u}(V) = 0$. This is a contradiction.
\end{proof}

\begin{lemma}\label{Lemma-Selection}
Let $n \geq 1$ be an integer, and let $\nu: K^n \to M_0(K^n)$ be a weak$^*$ continuous function that vanishes at infinity. Let $A$ and $B$ be uncountable subsets of $L^n$ such that the points of $A$ are pairwise disjoint and for all $y \in A$ and $z \in B$ it holds that $\mathrm{supp}(y) \cap \mathrm{supp}(z) = \emptyset$. Then, for every $\epsilon > 0$, there exist sequences $\{y_{\alpha}\}_{\alpha < \omega_1}$ and $\{z_{\alpha}\}_{\alpha < \omega_1}$ of pairwise distinct points of $A$ and $B$ respectively, and also sequences $\{G_{\alpha}\}_{\alpha < \omega_1}$ and $\{H_{\alpha}\}_{\alpha < \omega_1}$ of finite subsets of $L$ satisfying the following conditions:
\begin{itemize}
    \item[(1)] $\mathrm{supp}(y_\alpha) \subseteq G_{\alpha} \setminus H_{\alpha}$, $\mathrm{supp}(z_{\alpha}) \subseteq H_{\alpha} \setminus G_{\alpha}$,
    \item[(2)] $|\nu_{y_{\alpha}}|\left(K^n \setminus G_{\alpha}^n\right) < \frac{\epsilon}{3}$, $|\nu_{z_{\alpha}}|\left(K^n \setminus H_{\alpha}^n\right) < \frac{\epsilon}{3}$.
\end{itemize}
\end{lemma}
\begin{proof}
Fix $\epsilon > 0$. Since Radon measures on scattered spaces are atomic, for each $z \in B$, we can fix a finite set $E_{z} \subseteq K$ such that $\mathrm{supp}(z) \subset E_{z}$ and
\[
|\nu_{z}|(K^n \setminus E_{z}^n) < \frac{\epsilon}{6 (2^n - 1)}.
\]

According to the $\varDelta$-System Lemma, we may assume that $\{E_{z} : z \in B\}$ forms an uncountable $\varDelta$-system, where all sets $E_z$ have the same cardinality, with root $R$. Since the points of $A$ are pairwise disjoint, by reducing $A$ to an uncountable subset if necessary, we may assume that for each $y \in A$, there is a finite subset $F_y \subset L$ such that $\mathrm{supp}(y) \subset F_y$, $F_y \cap R = \emptyset$, and
\[
|\nu_{y}|(K^n \setminus (F_y \cup R)^n) < \frac{\epsilon}{3}.
\]

Now, observe that the product $(F_y \cup R)^n$ can be written as the union of pairwise disjoint sets $N_i(y)$, $0 \leq i \leq 2^n - 1$, where $N_0(y) = F_y^n$ and $\mathrm{supp}(y)\setminus \mathrm{supp}(w)\neq \emptyset$ for all $w \in N_i(y)$ for each $1 \leq i \leq 2^n - 1$. From Theorem \ref{Thm:vanish-B}, we may reduce $A$ to an uncountable subset and assume that $|\nu_{y}|(N_i(y)) = 0$ for all $y \in A$ and all $1 \leq i \leq 2^n - 1$. Hence, for each $y \in A$, the following holds:
\[
|\nu_{y}|(K^n \setminus F_y^n) < \frac{\epsilon}{3}.
\]

Given $\beta < \omega_1$, suppose that for each $\alpha < \beta$ we have obtained sequences $\{y_{\alpha}\}_{\alpha < \beta}$, $\{z_{\alpha}\}_{\alpha < \beta}$, $\{G_{\alpha}\}_{\alpha < \beta}$, $\{H_{\alpha}\}_{\alpha < \beta}$ satisfying items (1) and (2) above. Let $y_\beta$ be any point of $A \setminus \{y_\alpha : \alpha < \beta\}$ and define $G_\beta = F_{y_\beta}$. Since $\mathrm{supp}(y_{\beta}) \cap R = \emptyset$, the set $B' = \{z \in B : \mathrm{supp}(y_{\beta}) \cap E_{z} \neq \emptyset\}$ has cardinality at most $n$. We pick $z_{\beta} \in B \setminus (\{z_\alpha : \alpha < \beta\} \cup B')$ such that $\mathrm{supp}(z_\beta) \cap G_{\beta} = \emptyset$ and let $H_{\beta} \subset K$ be a finite set such that $\mathrm{supp}(z_{\beta}) \subset H_{\beta}$, $H_{\beta} \cap \mathrm{supp}(y_{\beta}) = \emptyset$, and
\[
|\nu_{z_{\beta}}|\left(K^n \setminus (H_{\beta} \cup \mathrm{supp}(y_{\beta}))^n\right) < \frac{\epsilon}{6}.
\]

Denoting $M_0 = H_{\beta}^n$, since $\mathrm{supp}(y_{\beta}) \cap (H_{\beta} \cup E_{z_{\beta}}) = \emptyset$, the product $(H_{\beta} \cup \mathrm{supp}(y_{\beta}))^n$ can be written as the union of pairwise disjoint sets $M_i$, $0 \leq i \leq 2^n - 1$, such that $M_i \cap E^n_{z_{\beta}} = \emptyset$ for $1 \leq i \leq 2^n - 1$. We have
\[
|\nu_{z_{\beta}}|(K^n \setminus H_{\beta}^n) < \frac{\epsilon}{6} + \sum_{i=1}^{2^n - 1} |\nu_{z_{\beta}}|(M_i) < \frac{\epsilon}{6} + (2^n - 1) \frac{\epsilon}{6(2^n - 1)} = \frac{\epsilon}{3},
\]
which completes the construction of the sequences.
\end{proof}

\begin{theorem}\label{Thm:vanish-A}
Let $n \geq 1$ be an integer, and let $\nu: K^n \to M_0(K^n)$ be a weak$^*$ continuous function that vanishes at infinity. Suppose $\mathcal{A} \subset L^n$ is an uncountable set consisting of pairwise disjoint points. Moreover, assume that there exists a partition $I_1, \ldots, I_k$ of $\{1, \ldots, n\}$ such that for every $y = (y_1, \ldots, y_n) \in \mathcal{A}$, the equality $y_s = y_t$ holds if and only if $s, t \in I_i$ for some $1 \leq i \leq k$. Then, there exist a real number $r \in \mathbb{R}$ and a countable subset $\mathcal{A}_0 \subset \mathcal{A}$ such that $\nu_{y}(\{y\}) = r$ for all $y \in \mathcal{A} \setminus \mathcal{A}_0$. 
\end{theorem}
\begin{proof}
Towards a contradiction, assume that for every $r \in\mathbb{R}$ and for every countable set $\mathcal{A}_0\subset \mathcal{A}$, there is  $y \in \mathcal{A}\setminus \mathcal{A}_{0}$ 
such that $\nu_{y}(\{y\})\neq r$. Then, according to \cite[Lemma 4.6]{candido2} there are rational numbers $p<q$ and uncountable sets $A,\ B\subset \mathcal{A}$ such that whenever $y \in A$ and $z \in B$ holds: \[\nu_{y}(\{y\})<p<q<\nu_{z}(\{z\}).\]

Since, for all $y \in A$ and $z \in B$, we have $\mathrm{supp}(y) \cap \mathrm{supp}(z) = \emptyset$, we may apply Lemma \ref{Lemma-Selection} and obtain uncountable sequences of points $\{y_\alpha\}_{\alpha < \omega_1} \subset A$ and $\{z_\alpha\}_{\alpha < \omega_1} \subset B$, as well as sequences $\{G_\alpha\}_{\alpha < \omega_1}$ and $\{H_\alpha\}_{\alpha < \omega_1}$ of finite subsets of $L$ satisfying:
\begin{itemize}
\item[(1)] $\mathrm{supp}(y_{\alpha})\subseteq G_{\alpha}\setminus H_{\alpha}$, $\mathrm{supp}(z_{\alpha})\subseteq H_{\alpha}\setminus G_{\alpha}$.
\item[(2)] $|\nu_{y_{\alpha}}|\left(K^n\setminus G_{\alpha}^n\right)<\frac{q-p}{3}$, $|\nu_{z_{\alpha}}|\left(K^n\setminus H_{\alpha}^n\right)<\frac{q-p}{3}$.
\end{itemize}

From the $\varDelta$-System Lemma, we can assume that $\{G_\alpha \cup H_\alpha : \alpha < \omega_1\}$ forms a $\varDelta$-system with root $R$, where all sets $G_\alpha \cup H_\alpha$ have the same cardinality. For each $\alpha < \omega_1$, let $G'_\alpha = G_\alpha \setminus R$ and $H'_\alpha = H_\alpha \setminus R$, and assume that $|G'_\alpha| = l$ and $|H'_\alpha| = m$. We further assume that either $l > 1$ or $m > 1$, as otherwise the proof can be completed with minor adjustments (see Remark \ref{Rem:Singleton}). If $y_\alpha = (y_1^\alpha, \ldots, y_n^\alpha)$ and $z_\alpha = (z_1^\alpha, \ldots, z_n^\alpha)$, denote
$G^\prime_\alpha = \{x_1^\alpha, \ldots, x_l^\alpha\}$ and $H^\prime_\alpha = \{x_{l+1}^\alpha, \ldots, x_{l+m}^\alpha\}$, where $x_i^\alpha = y_j^\alpha$ 
and $x_{l+i}^\alpha = z_j^\alpha$ if and only if $j \in I_i$ for $1 \leq i \leq k$ and $1 \leq j \leq n$.

The collection $\mathcal{C} = \{(x_1^\alpha, \ldots, x_{m+l}^\alpha) : \alpha < \omega_1\} \subset L^{m+l}$, constructed from the sets $G^\prime_\alpha$ and $H^\prime_\alpha$, forms an uncountable $(\{1, l+1\}, \ldots, \{k, l+k\}, \{k+2, \ldots, l, l+k+1, \ldots, l+m\})$-separated subset of $L^{m+l}$. We can fix an ordinal $\xi < \omega_1$ such that $R \subset L_\xi$. By removing countably many elements from $\mathcal{C}$ if necessary, we can assume that $\mathcal{C} \subset (L \setminus L_\xi)^{m+l}$. Since $L$ possesses the collapsing property, there exist distinct points $a_1, \ldots, a_k \in L \setminus R$ such that $b = (b_1, \ldots, b_{l+m})$, where $b_j = a_i$ if $j \in \{i, l+i\}$ for $1 \leq i \leq k$ and $b_j = \infty$ otherwise, is an accumulation point of $\mathcal{C}$. We then fix a net $\{(x_1^{\alpha_\gamma}, \ldots, x_{m+l}^{\alpha_\gamma})\}_{\gamma \in \Gamma}$ in $\mathcal{C}$ that converges to the point $b$.

Next, we select pairwise disjoint clopen compact neighborhoods $U_1, \ldots, U_k$ around $a_1, \ldots, a_k$, respectively, ensuring that $(U_1 \cup \ldots \cup U_k) \cap R = \emptyset$. By passing to a subnet if necessary, we can assume that $U_i \cap G_{\alpha_{\gamma}} = \{ x^{\alpha_{\gamma}}_{i} \}$ and $U_i \cap H_{\alpha_{\gamma}} = \{ x^{\alpha_{\gamma}}_{l+i} \}$ for all $1 \leq i \leq k$ and $\gamma \in \varGamma$.

For each $1 \leq j \leq n$, we then set $V_j = U_i$ and $w_j = a_i$ whenever $j \in I_i$ for some $1 \leq i \leq k$. Finally, we define $V=V_1 \times \ldots \times V_n$ and fix the point $w = (w_1, \ldots, w_n)$.

Let $\{u_{\gamma}\}_{\gamma \in \varGamma}$ be the net where, for each $\gamma$, $u_{\gamma} = (u^{\gamma}_1, \ldots, u^{\gamma}_n)$ is such that $u^{\gamma}_j = x^{\alpha_{\gamma}}_i$ if $j \in I_i$ for some $1 \leq i \leq k$. Similarly, let $\{v_{\gamma}\}_{\gamma \in \varGamma}$ be the net where, for each $\gamma$, $v_{\gamma} = (v^{\gamma}_1, \ldots, v^{\gamma}_n)$ is such that $v^{\gamma}_j = x^{\alpha_{\gamma}}_{l+i}$ if $j \in I_i$ for some $1 \leq i \leq k$. It is straightforward to verify that $G_{\alpha_{\gamma}}^n \cap V = \{u_\gamma\}$ and $H_{\alpha_{\gamma}}^n \cap V = \{v_\gamma\}$ for each $\gamma \in \varGamma$, and that both nets converge to $w$. Since $u_{\gamma} \in A$ and $v_{\gamma} \in B$ for each $\gamma \in \varGamma$, we can write:
\begin{align*}
\nu_{u_{\gamma}}(V)&=|\nu_{u_{\gamma}}|(\{u_{\gamma}\})+|\nu_{u_{\gamma}}|(V\setminus \{u_{\gamma}\})\\
&\leq |\nu_{u_{\gamma}}|(\{u_{\gamma}\})+|\nu_{u_{\gamma}}|(K^n\setminus G_{\alpha_{\gamma}}^n)<p+\frac{q-p}{3}=\frac{2p+q}{3},
\end{align*}
\begin{align*}
\nu_{v_{\gamma}}(V)&=\nu_{v_{\gamma}}(\{v_{\gamma}\})-|\nu_{v_{\gamma}}|(V\setminus \{v_{\gamma}\})\\
&\geq \nu_{v_{\gamma}}(\{v_{\gamma}\})-|\nu_{v_{\gamma}}|(K^n\setminus H_{\alpha_{\gamma}}^n)>q-\frac{q-p}{3}=\frac{2q+p}{3}.
\end{align*}

Recalling that $\nu$ is a weak$^*$-continuous function and $V$ is a clopen compact subset of $L^n$, the above relations imply that
\begin{align*}
\nu_w(V) &= \lim_{\gamma\to \infty}\nu_{u_\gamma}(V) < \frac{2p+q}{3}< \frac{2q+p}{3} < \lim_{\gamma \to \infty}\nu_{v_\gamma}(V) = \nu_w(V),
\end{align*}
which is a contradiction.
\end{proof}

\begin{corollary}\label{Cor:vanish-A}
Let $n \geq 1$ be an integer, and let $\nu: K^n \to M_0(K^n)$ be a weak$^*$ continuous function that vanishes at infinity. There exists $r \in \mathbb{R}$ and $\rho<\omega_1$ such that $\nu_y(\{y\}) = r$ for all $y \in (L\setminus L_\rho)^n\setminus \varDelta(K^n)$.
\end{corollary}
\begin{proof}
Assume, for the sake of contradiction, that the statement is false. Then for every $r \in \mathbb{R}$ and every $\alpha < \omega_1$, there exists a point $y \in (L \setminus L_\alpha)^n \setminus \varDelta(K^n)$ such that $\nu_y(\{y\}) \neq r$. Let $\beta < \omega_1$ be given. Suppose we have already constructed a set $\{y_{\alpha} : \alpha < \beta\} \subseteq (L \setminus L_\rho)^n \setminus \varDelta(K^n)$. Choose $\gamma_{\beta} < \omega_1$ such that $\{y_{\alpha} : \alpha < \beta\} \subseteq L_{\gamma_{\beta}}^n \setminus \varDelta(K^n)$. If $\beta$ is a successor ordinal, i.e., $\beta = \eta + 1$, select $y_{\beta}$ in $(L \setminus L_{\gamma_{\beta}})^n \setminus \varDelta(K^n)$ such that $\nu_{y_{\beta}}(\{y_{\beta}\}) \neq \nu_{y_{\eta}}(\{y_{\eta}\})$. If $\beta$ is not a successor ordinal, choose $y_{\beta}$ to be any point in $(L \setminus L_{\gamma_{\beta}})^n \setminus \varDelta(K^n)$.

By induction, we have constructed a collection $\mathcal{A} = \{y_{\alpha} : \alpha < \omega_1\}$ of pairwise disjoint points in $L^n \setminus \varDelta(K^n)$. This contradicts Theorem \ref{Thm:vanish-A}, because for any $\alpha < \omega_1$, the construction guarantees that $\nu_{y_{\alpha}}(\{y_{\alpha}\}) \neq \nu_{y_{\alpha + 1}}(\{y_{\alpha + 1}\})$.
\end{proof}

%%%%%%%%%%%%%%%%%%%%%%%%%%%%%%%%%%%%%%%%%%%%%%%%%%%
%%%%%%%%%%%%%%%%%%%%%%%%%%%%%%%%%%%%%%%%%%%%%%%%%%%

For a natural number $n \geq 1$, we denote by $\mathfrak{S}_n$ the group of all permutations of the set $\{1, 2, \ldots, n\}$. Given an element $x = (x_1, \ldots, x_n) \in K^n$, we write $\sigma \cdot x = (x_{\sigma(1)}, \ldots, x_{\sigma(n)})$ to represent the action of a permutation $\sigma \in \mathfrak{S}_n$ on $x$.

\begin{theorem}\label{Thm:vanish-C}
Let $n \geq 1$ be an integer, and let $\nu: K^n \to M_0(K^n)$ be a weak$^*$-continuous function that vanishes at infinity. Suppose $A = \{y_{\alpha} : \alpha < \omega_1\}$ and $B = \{z_{\alpha} : \alpha < \omega_1\}$ are subsets of $K^n$ such that the points in $A$ are pairwise disjoint and $\mathrm{supp}(y_{\alpha}) \subset \mathrm{supp}(z_{\alpha})$ for all $\alpha < \omega_1$. Assume there exists $\xi < \omega_1$ such that $\nu_x(\{\sigma \cdot x\}) = 0$ for every $x \in (L \setminus L_\xi)^n \setminus \varDelta(K^n)$ and every permutation $\sigma \in \mathfrak{S}_n$. Then there exists $\rho < \omega_1$ such that $\nu_{y_{\alpha}}(\{z_{\alpha}\}) = 0$ for all $\rho < \alpha < \omega_1$.
\end{theorem}
\begin{proof}
If the statement were false, then by passing to an appropriate uncountable subset of indices, we may assume that there exists $\epsilon > 0$ such that $|\nu_{y_\alpha}|(\{z_\alpha\}) \geq \epsilon$ for all $\alpha < \omega_1$, and the collection $\{\mathrm{supp}(z_\alpha) : \alpha < \omega_1\}$ forms a $\varDelta$-system with root $R$.

We can also assume the existence of a partition $\{I_1, \ldots, I_k\}$ of $\{1, \ldots, n\}$ such that, for each $\alpha < \omega_1$, if $y_\alpha = (y^\alpha_1, \ldots, y^\alpha_n)$, then $y^\alpha_i = y^\alpha_j$ if and only if $i$ and $j$ belong to the same set $I_r$ for some $1 \leq r \leq k$. Similarly, there exists a partition $\{F_1, \ldots, F_{l+1}\}$ such that, for each $\alpha < \omega_1$, if $z_\alpha = (z^\alpha_1, \ldots, z^\alpha_n)$, then $R = \{z^\alpha_j : j \in F_{l+1}\}$. For all $i, j \in \{1, \ldots, n\} \setminus F_{l+1}$, we have $z^\alpha_i = z^\alpha_j$ if and only if $i$ and $j$ belong to the same set $F_s$ for some $1 \leq s \leq l$. Additionally, we assume that there exists a sequence $\{c_j\}_{j \in F_{l+1}}$ such that $z^\alpha_j = c_j$ for every $j \in F_{l+1}$ for all $\alpha < \omega_1$. If $R = \emptyset$, we set $F_{l+1} = \emptyset$ and proceed with the proof with minor adjustments.

Since $\mathrm{supp}(y_\alpha) \subset \mathrm{supp}(z_\alpha)$, by relabeling the partitions if necessary, we may suppose that $\{y^\alpha_j : j \in I_i\} = \{z^\alpha_j : j \in F_i\}$ for each $1 \leq i \leq k$.

By passing to a further uncountable subset of indices, we can fix a collection $\{x_{\alpha} : \alpha < \omega_1\}$ of pairwise distinct points of $(L \setminus L_\xi)^n \setminus \varDelta(K^n)$ such that $\mathrm{supp}(x_{\alpha}) \cap \mathrm{supp}(z_\beta) = \emptyset$ whenever $\alpha \neq \beta$. For each $\alpha < \omega_1$ and each $1 \leq j \leq n$, if $j \in F_i$ for some $1 \leq i \leq k$, we define $M_j(\alpha) = \{x^{\alpha}_{j} : j \in I_i\}$. If $j \in F_i$ for some $k+1 \leq i \leq l+1$, we set $M_j(\alpha) = \emptyset$. Applying Theorem \ref{Thm:vanish-B} and recalling that $\nu_{x}(\{\sigma \cdot x\}) = 0$ for all $x \in (L \setminus L_\xi)^n \setminus \varDelta(K^n)$ and $\sigma \in \mathfrak{S}_n$, we can also suppose that $|\nu_{x_\alpha}|(M_1(\alpha) \times \dots \times M_n(\alpha)) = 0$ for every $\alpha < \omega_1$.

Employing Lemma \ref{Lemma-Selection}, we may pass to an even further uncountable subset and assume that there are sequences $\{G_{\alpha}\}_{\alpha \in \omega_1}$ and $\{H_{\alpha}\}_{\alpha \in \omega_1}$ of finite subsets of $K$ satisfying:
\begin{itemize}
    \item[(1)] $\mathrm{supp}(x_{\alpha}) \subseteq G_{\alpha} \setminus H_{\alpha}$ and $\mathrm{supp}(z_{\alpha}) \subseteq H_{\alpha} \setminus G_{\alpha}$,
    \item[(2)] $|\nu_{x_{\alpha}}|\left(K^n \setminus G_{\alpha}^n\right) < \frac{\epsilon}{3}$ and $|\nu_{z_{\alpha}}|\left(K^n \setminus H_{\alpha}^n\right) < \frac{\epsilon}{3}$.
\end{itemize}
Moreover, by the $\varDelta$-System Lemma, we may assume that the family $\{G_\alpha \cup H_\alpha : \alpha < \omega_1\}$ forms a $\varDelta$-system with root $Q$, where $R \subseteq Q$, and that all sets in this family have the same cardinality. Let $G'_\alpha = G_\alpha \setminus Q$ and $H'_\alpha = H_\alpha \setminus Q$, and assume that $|G'_\alpha| = m$ and $|H'_\alpha| = r$ for each $\alpha < \omega_1$.

We denote $G^\prime_\alpha = \{s^\alpha_1, \ldots, s^\alpha_m\}$ and $H^\prime_\alpha = \{s^\alpha_{m+1}, \ldots, s^\alpha_{m+r}\}$, and further assume that, for all $\alpha < \omega_1$, the first $n$ elements of $G_\alpha$ correspond to the coordinates of the point $x_\alpha$. That is, $s_j^\alpha = x^\alpha_j$ for each $1 \leq j \leq n$. Additionally, for each $1 \leq j \leq l$, we suppose $s^\alpha_{m+j} = z^\alpha_i$ if and only if $i \in F_j$.

We define the set $\mathcal{C} = \{(s_1^\alpha, \ldots, s_{m+r}^\alpha) : \alpha < \omega_1\} \subset K^{m+r}$ and observe that it forms a $\big((I_1 \cup \{m+1\}), \ldots, (I_k \cup \{m+k\}), \{m+k+1\}, \ldots, \{m+l\}, \{n+1, \ldots, m, m+l+1, \ldots, m+r\}\big)$-separated subset of $L^{m+r}$. Without loss of generality, we may assume that $Q^n \subset L_{\xi}^n$ and $\mathcal{C} \subset (L \setminus L_\xi)^n$. Since $L$ has the collapsing property, there exist pairwise distinct points $a_1, \ldots, a_l \in L \setminus Q$ such that the point $b = (b_1, \ldots, b_{m+r})$, where $b_j = a_i$ if $j \in (I_i \cup \{m+i\})$ for some $1 \leq i \leq k$, $b_{m+j} = a_j$ for $k+1 \leq j \leq l$, and $b_j = \infty$ otherwise, is an accumulation point of $\mathcal{C}$. We then fix a net $\{(s^{\alpha_{\gamma}}_1, \ldots, s^{\alpha_{\gamma}}_{m+r})\}_{\gamma \in \varGamma}$ in $\mathcal{C}$ that converges to $b$.

We let $U_1, U_2, \ldots, U_{l+1}$ be pairwise disjoint clopen compact subsets of $L$ such that $a_i \in U_i$ for each $1 \leq i \leq l$, and $R \subset U_{l+1}$. We also assume that $\left(U_1 \cup U_2 \cup \cdots \cup U_{l+1}\right) \cap (Q \setminus R) = \emptyset$. Therefore, by passing to a subnet if necessary, we may assume that for each $\gamma \in \varGamma$, the following conditions hold: $U_i \cap G_{\alpha_{\gamma}} = \{s^{\alpha_{\gamma}}_j : j \in I_i\}$ and $U_i \cap H_{\alpha_{\gamma}} = \{s^{\alpha_{\gamma}}_{m+i}\}$ for all $1 \leq i \leq k$; $U_i \cap G_{\alpha_{\gamma}} = \emptyset$ and $U_i \cap H_{\alpha_{\gamma}} = \{s^{\alpha_{\gamma}}_{m+i}\}$ for all $k+1 \leq i \leq l$; and $U_{l+1} \cap G_{\alpha_{\gamma}} = \emptyset$ while $U_{l+1} \cap H_{\alpha_{\gamma}} = R$.

For each $1 \leq j \leq n$, set $V_j = U_i$ and $w_j = a_j$ whenever $j \in F_i$ for some $1 \leq i \leq l$. Recalling that 
$R=\{c_j:j\in F_{l+1}\}$ we define $V_j = U_{l+1}$ and $w_j=c_j$ for each $j \in F_{l+1}$. Define $w = (w_1, \ldots, w_n)$ and consider the set $P = \{(f_1, \ldots, f_n) : f_j = w_j \text{ if } j \notin F_{l+1}, \text{ and } f_j \in R \text{ otherwise}\}$. Next, let $V \subset V_1 \times \cdots \times V_n$ be a clopen neighborhood of $w$ such that $V \cap P = \{w\}$.

From the net above, we form the following three nets:
\begin{itemize}
    \item $\{u_{\gamma}\}_{\gamma \in \varGamma}$ where $u_{\gamma}=(u^{\gamma}_1,\ldots,u^{\gamma}_n)$ is given by $u^{\gamma}_j = s^{\alpha_{\gamma}}_j$ for all $1 \leq j \leq n$,
    \item $\{v_{\gamma}\}_{\gamma \in \varGamma}$ where $v_{\gamma}=(v^{\gamma}_1,\ldots,v^{\gamma}_n)$ given by $v^{\gamma}_j = s^{\alpha_{\gamma}}_{m+i}$ if $j \in I_i$ for some $1 \leq i \leq k$,
    \item $\{w_{\gamma}\}_{\gamma \in \varGamma}$ where $w_{\gamma}=(w^{\gamma}_1,\ldots,w^{\gamma}_n)$ given by $w^{\gamma}_j = s^{\alpha_{\gamma}}_{m+i}$ if $j \in F_i$ for some $1 \leq i \leq l$, and $w^{\gamma}_t=c_j$ for every $j \in F_{l+1}$.
\end{itemize}
From the construction of $V$ and since $\{w_{\gamma}\}_{\gamma \in \varGamma}$ coverges to $w$, by passing to a subnet, we may assume that for each $\gamma \in \varGamma$, the following holds:
\[
G_{\alpha_\gamma}^n \cap V \subset M_1(\alpha_\gamma) \times \cdots \times M_n(\alpha_\gamma), \quad \text{and} \quad H_{\alpha_\gamma}^n \cap V = \{w_{\gamma}\}.
\]

Hence, recalling that $|\nu_{u_{\gamma}}|(M_1(\alpha_\gamma) \times \ldots \times M_n(\alpha_\gamma)) = 0$ for every $\gamma \in \varGamma$, it follows that:

\begin{align*}
|\nu_{u_{\gamma}}(V)|&\leq |\nu_{u_{\gamma}}|(M_1(\alpha_\gamma)\times \ldots\times M_n(\alpha_\gamma))|+|\nu_{u_{\gamma}}|(K^n\setminus G_{\alpha_{\gamma}}^n)|\leq\frac{\epsilon}{3}=\frac{\epsilon}{3}
\end{align*}
and
\begin{align*}
|\nu_{v_{\gamma}}(V)|\geq |\nu_{v_{\gamma}}|(\{w_{\gamma}\})-|\nu_{v_{\gamma}}|(K^n\setminus H_{\alpha_{\gamma}}^n)\geq\epsilon-\frac{\epsilon}{3}=\frac{2\epsilon}{3}.
\end{align*}

Recalling that the net $\{(s^{\alpha_{\gamma}}_1, \ldots, s^{\alpha_{\gamma}}_{m+r})\}_{\gamma \in \varGamma}$ converges to $b = (b_1, \ldots, b_{m+r})$, we define $q = (q_1, \ldots, q_n)$ by setting $q_j = b_i$ whenever $j \in I_i$ for some $1 \leq i \leq k$. It is clear that both nets $\{u_{\gamma}\}_{\gamma \in \varGamma}$ and $\{v_{\gamma}\}_{\gamma \in \varGamma}$ converge to $q$.  Since $\nu$ is weak$^*$-continuous and $V$ is a clopen compact subset of $L^n$, we have
\[
|\nu_q(V)| = \lim_{\gamma \to \infty} |\nu_{u_{\gamma}}(V)| \leq \frac{\epsilon}{3} < \frac{2\epsilon}{3} \leq \lim_{\gamma \to \infty} |\nu_{v_{\gamma}}(V)| = |\nu_q(V)|.
\]
This is a contradiction.

\end{proof}

\section{Operators on $C_0(L\times L)$}
\label{Sec-Operators}

As in the previous section, we continue to assume that $L$ is the space obtained from Theorem \ref{compactclub}, with $K = L \cup \{\infty\}$ denoting its Aleksandrov one-point compactification. The primary goal in this section is to provide a thorough description of all operators on the space $C_0(L \times L)$ and to establish Theorem \ref{main1}. Our approach involves decomposing an operator $T: C_0(L \times L) \to C_0(L \times L)$ into simpler components through a three-stage process. In the first stage, we isolate a component of the operator $T$ that is induced by permutations of the coordinates of a function in $C_0(L \times L)$. Specifically, let $\mathfrak{S}_n$ denote the set of all permutations of $\{1, 2, \ldots, n\}$. For each $\sigma \in \mathfrak{S}_n$, we define the operator $I_{\sigma}: C_0(L^n) \to C_0(L^n)$ by the formula:
\[I_\sigma(f)(x_1, \ldots, x_n) = f(x_{\sigma(1)}, \ldots, x_{\sigma(n)}).\]
Regardless of the topology on the locally compact space $L$, operators of this form are always well-defined. We will refer to any operator obtained as a linear combination of the operators $I_\sigma$ as a \emph{natural operator}. The first step in the decomposition is captured by the following result:

\begin{theorem}\label{ReducaoOperador}
For every operator $T:C_0(L^n)\to C_0(L^n)$ there is a set of scalars $\{a_\sigma:\sigma\in \mathfrak{S}_n\}\subseteq \mathbb{R}$ and an ordinal $\alpha<\omega_1$ such that if $S=T-\sum_{\sigma\in \mathfrak{S}_n}a_\sigma I_\sigma$, then 
 \[S[C_0(L^n)]\subseteq C_0(L^n\setminus (L\setminus L_{\alpha})^n).\]
\end{theorem}
\begin{proof}
For each $\sigma \in \mathfrak{S}_n$, the map $y \mapsto (T \circ I_{\sigma^{-1}})^*(\delta_y)$ defines a weak$^*$ continuous function from $K^n$ to $M_0(K^n)$ that vanishes at infinity. According to Corollary \ref{Cor:vanish-A}, there exist $a_\sigma \in \mathbb{R}$ and $\xi_\sigma < \omega_1$ such that
\[ 
(T \circ I_{\sigma^{-1}})^*(\delta_x)(\{x\}) = T^*(\delta_x)(\{\sigma \cdot x\}) = a_\sigma 
\]
for all $x \in (L \setminus L_{\xi_\sigma})^n\setminus \varDelta(K^n)$.

We let $\xi = \max\{\xi_\sigma : \sigma \in \mathfrak{S}_n\}$ and consider the operator $S = T - \sum_{\sigma \in \mathfrak{S}_n} a_\sigma I_\sigma$. We note that for every $x \in (L \setminus L_\xi)^n\setminus \varDelta(K^n)$ and $\gamma \in \mathfrak{S}_n$, we have
\begin{align*}
    S^*(\delta_x)(\{\gamma\cdot x\}) &= T^*(\delta_x)(\{\gamma\cdot x\}) - \sum_{\sigma \in \mathfrak{S}_n} a_\sigma I_\sigma^*(\delta_x)(\{\gamma\cdot x\}) \\
    &= T^*(\delta_x)(\{\gamma \cdot x\}) - \sum_{\sigma \in \mathfrak{S}_n} a_\sigma  \delta_{\sigma \cdot x}(\{\gamma\cdot x)\}) \\
    &= T^*(\delta_x)(\{\gamma \cdot x\}) - a_\gamma = 0.
\end{align*}

We claim that there exists some $\alpha < \omega_1$ such that $S[C_0(L^n)] \subset C_0(L^n \setminus (L \setminus L_{\alpha})^n)$. Assume, for the sake of contradiction, that for each $\alpha < \omega_1$, there is a function $f_\alpha \in C_0(L^n)$ and a point $y_\alpha \in (L \setminus L_\alpha)^n$ such that $S(f_\alpha)(y_\alpha) \neq 0$. Since 
\[S(f_\alpha)(y_\alpha) = S^*(\delta_{y_\alpha})(f_\alpha)=\sum_{z\in L^n}f_\alpha(z)S^*(\delta_{y_\alpha})(\{z\}),\] 
there exists $z_\alpha \in L^n$ with $S^*(\delta_{y_\alpha})(\{z_\alpha\}) \neq 0$. By passing to an uncountable subset of indices and relabeling if necessary, we may assume that the points in the set $\{y_\alpha : \alpha < \omega_1\}$ are pairwise disjoint. If $\mathrm{supp}(y_\alpha)\setminus\mathrm{supp}(z_\alpha) \neq \emptyset$ for uncountably many values of $\alpha$, this contradicts Theorem \ref{Thm:vanish-B}. On the other hand, if $\mathrm{supp}(y_\alpha) \subset \mathrm{supp}(z_\alpha)$ for uncountably many $\alpha$, recalling that $S^*(\delta_x)(\{\sigma \cdot x\}) = 0$ for all $x \in (L \setminus L_\xi)^n\setminus \varDelta(K^n)$, we arrive at a contradiction with Theorem \ref{Thm:vanish-C}.
\end{proof}

To simplify notation, we will henceforth consider $I: C_0(L \times L) \to C_0(L \times L)$ as the identity operator and $J: C_0(L \times L) \to C_0(L \times L)$ as the operator that transposes the coordinates of a function, defined by $J(f)(x, y) = f(y, x)$ for every $f \in C_0(L \times L)$.

In the space $C_0(L \times L)$, beyond the natural operators, which are linear combinations of the operators $I$ and $J$, there exist other elementary operators that arise through a specific form of multiplication by a function in $C_0(L)$. This process is formalized in the following definition:

\begin{definition}\label{Def:DiagonalMulti}
For given functions $g, h \in C_0(L)$, we define the diagonal multiplication operators $M_g,\ N_h: C_0(L \times L) \to C_0(L \times L)$ as follows:
\[ M_g(f)(x,y) = g(y)f(x,x) \quad \text{and} \quad N_h(f)(x,y) = h(x)f(y,y). \]
\end{definition}

\begin{lemma}\label{Lem:Vanish-D}
Let $\nu: K^2 \to M_0(K^2)$ be a weak$^*$ continuous function that vanishes at infinity. For every $\alpha<\omega_1$, there exists $\alpha \leq \beta < \omega_1$ and functions $g,h \in C_0(L_\beta)$ such that, for all $x, y \in L \setminus L_\beta$ and $a \in L_\beta$, we have
\begin{displaymath}
\nu_{(x,a)}(\{(x,y)\})=\nu_{(x,a)}(\{(y,x)\})  = \left\{
\begin{array}{ll}
0 & \text{if } x \neq y,\\
g(a) & \text{if } x = y.
\end{array}
\right.
\end{displaymath}
\begin{displaymath}
\nu_{(a,y)}(\{(x,y)\}) =\nu_{(a,y)}(\{(x,y)\})= \left\{
\begin{array}{ll}
0 & \text{if } x \neq y,\\
h(a) & \text{if } x = y.
\end{array}
\right.
\end{displaymath}
\end{lemma}
\begin{proof}
Let $a \in L$ be arbitrary, and define the functions $P_1, P_2: K^2 \to K^2$ by $P_1(x, y) = (x, a)$ and $P_2(x, y) = (a, y)$. Applying Theorem \ref{Thm:vanish-A} and Corollary \ref{Cor:vanish-A} to each of the functions $\nu \circ P_1$, $J^* \circ \nu \circ P_1$, $\nu \circ P_2$, and $J^* \circ \nu \circ P_2$, we obtain $\rho_a < \omega_1$ and scalars $r_a, g_a, s_a, u_a, h_a, v_a \in \mathbb{R}$ such that for every $x, y \in L \setminus L_{\rho_a}$ with $x \neq y$, the following holds:
\[
\begin{aligned}
\nu_{(x,a)}(\{(x,y)\}) &= r_a, & \nu_{(a,y)}(\{(x,y)\}) &= u_a, \\
\nu_{(x,a)}(\{(x,x)\}) &= g_a, & \nu_{(a,y)}(\{(y,y)\}) &= h_a, \\
\nu_{(x,a)}(\{(y,x)\}) &= s_a, & \nu_{(a,y)}(\{(y,x)\}) &= v_a.
\end{aligned}
\]

Now, observe that for every $n \in \mathbb{N}$, if $\{x, y_1, \ldots, y_n\}$ is a subset of $n+1$ distinct points from $L \setminus L_{\rho_a}$, then 
\[
n |r_a| = \sum_{i=1}^n |\nu_{(x,a)}|(\{(x,y_i)\}) = |\nu_{(x,a)}|(\{x\} \times \{y_1, \ldots, y_n\}) \leq |\nu_{(x,a)}|(K^2) < \infty,
\]
which implies that $r_a = 0$. Similarly, we can deduce that $r_a = s_a = u_a = v_a = 0$.

Next, let $\alpha < \omega_1$ be an arbitrary ordinal, and let $\{\beta_n\}_n$ be a sequence defined recursively as follows: set $\beta_0 = \alpha$, and assuming $\beta_0, \ldots, \beta_n < \omega_1$ have been constructed, define $\beta_{n+1} = \sup \{\rho_a : a \in L_{\beta_n}\}$. We are done by setting $\beta = \sup \{\beta_n : n < \omega\}$ and defining functions $g, h: L \to \mathbb{R}$ by 
\[
g(a) =
\begin{cases}
g_a, & \text{if } a \in L_\beta, \\
0, & \text{if } a \notin L_\beta,
\end{cases}
\quad
h(a) =
\begin{cases}
h_a, & \text{if } a \in L_\beta, \\
0, & \text{if } a \notin L_\beta.
\end{cases}
\]
Indeed, if $a \in L_\beta$, then there exists some $n < \omega_1$ such that $a \in L_{\beta_n}$, implying $\beta \geq \beta_{n+1} \geq \rho_a$.  Since $L \setminus L_\beta \subset L \setminus L_{\beta_{\rho_a}}$, it is clear that for $x, y \in L \setminus L_\beta$, the formulas in the statement of the lemma are satisfied.

To verify that $g, h \in C_0(L_\beta)$, pick an isolated point $z_0 \in L \setminus L_\beta$ and note that $g(a) = \nu_{(z_0,a)}(\{(z_0,z_0)\})$ and $h(a) = \nu_{(a,z_0)}(\{(z_0,z_0)\})$ for every $a \in L_\beta$.
\end{proof}

In the second stage of our decomposition, we will focus on isolating the components of the operator that correspond to the diagonal multiplications as defined earlier. Specifically, this step is formalized in the following theorem:

\begin{theorem}\label{ReducaoOperadorL2}
For every operator $T: C_0(L \times L) \to C_0(L \times L)$ there exist scalars $p$ and $q$, an ordinal $\beta < \omega_1$, and functions $g, h \in C_0(L_\beta)$ such that, if we define $Q = T - p I - q J - M_g - N_h$, then:
\begin{itemize}
\item $Q[C_0((L \setminus L_\beta)^2)] = \{\vec{0}\}$
\item $Q[C_0(L^2 \setminus (L \setminus L_\beta)^2)] \subseteq C_0(L^2 \setminus (L \setminus L_\beta)^2)$
\end{itemize}
\end{theorem}
\begin{proof}
From Theorem \ref{Thm:OperatorCharacterization}, there exist constants $p$ and $q$ and an ordinal $\alpha < \omega_1$ such that if $S = T - p  I - q  J$, then $S[C_0(L^2)] \subseteq C_0(L^2 \setminus (L \setminus L_{\alpha})^2)$. By applying Lemma \ref{Lem:Vanish-D} to the function $(x, y) \mapsto S^*(\delta_{(x, y)})$, we find $\alpha \leq \rho < \omega_1$ such that there exist functions $g, h \in C_0(L_\rho)$ satisfying, for every $a, b \in L_\rho$ and $x, y \in L \setminus L_\rho$, the following relations:
\[
S^*(\delta_{(x, a)})(\{(x, y)\}) = S^*(\delta_{(x, a)})(\{(x, y)\}) = \left\{
\begin{array}{ll}
0 & \text{if } x \neq y,\\
g(a) & \text{if } x = y.
\end{array}
\right.
\]
\[
S^*(\delta_{(b, x)})(\{(x, y)\}) = S^*(\delta_{(b, x)})(\{(y, x)\}) = \left\{
\begin{array}{ll}
0 & \text{if } x \neq y,\\
h(b) & \text{if } x = y.
\end{array}
\right.
\]
Define $Q = S - M_g - N_h$. It is clear that $Q[C_0(L^2)] \subseteq C_0(L^2 \setminus (L \setminus L_{\rho})^2)$. We claim that there exists $\rho \leq \beta < \omega_1$ such that $Q[C_0((L \setminus L_{\beta})^2)] = \{\vec{0}\}$.

Suppose, on the contrary, that for every $\rho \leq \beta < \omega_1$, there exist $f_\beta \in C_0((L \setminus L_\beta)^2)$ and $(z_\beta, w_\beta) \in L^2 \setminus (L \setminus L_{\rho})^2$ such that $Q(f_\beta)(z_\beta, w_\beta) \neq 0$. Since $L_\rho$ is countable, by passing to an uncountable subset and relabeling the elements, we may assume that $w_\beta = a$ for all $\beta < \omega_1$ (similarly, we could establish the case where $z_\beta = b$ for all $\rho \leq \beta < \omega_1$). Since $Q(f_\beta)(z_\beta, a) = Q^*(\delta_{(z_\beta, a)})(f_\beta)$, we can find $(x_\beta, y_\beta) \in (L \setminus L_\beta)^2$ such that $Q^*(\delta_{(z_\beta, a)})(\{(x_\beta, y_\beta)\}) \neq 0$. Passing to a further uncountable subset if necessary, we may assume that the points of $\{(x_\beta, y_\beta) : \beta < \omega_1\}$ are pairwise disjoint. If $z_\beta \notin \{x_\beta, y_\beta\}$ for uncountably many $\beta < \omega_1$, since $Q^*(\delta_{(z_\beta, a)})(\{(x_\beta, y_\beta)\}) \neq 0$ for each $\beta$, we have a contradiction with Theorem \ref{Thm:vanish-B}. Thus, we may suppose that $\{z_\beta\} \subset \{x_\beta, y_\beta\}$ for all $\beta$. If $x_\beta \neq y_\beta$ for some $\beta < \omega_1$, then $Q^*(\delta_{(z_\beta, a)})(\{(x_\beta, y_\beta)\}) = 0$, which is a contradiction. Therefore, $x_\beta = y_\beta = z_\beta$. In this case, we have
\[
Q^*(\delta_{(x_\beta, a)})(\{(x_\beta, x_\beta)\}) = S^*(\delta_{(x_\beta, a)})(\{(x_\beta, x_\beta)\}) - M_g^*(\delta_{(x_\beta, a)})(\{(x_\beta, x_\beta)\}) = g(a) - g(a) = 0,
\]
which is also a contradiction. Hence, our claim is established.

We then let $\rho \leq \beta < \omega_1$ be such that $Q[C_0((L \setminus L_\beta)^2)] = \{\vec{0}\}$. Consequently, we have:
\[Q[C_0(L^2 \setminus (L \setminus L_{\beta})^2)] = Q[C_0(L^2)] \subseteq C_0(L^2 \setminus (L \setminus L_\alpha)^2) \subseteq C_0(L^2 \setminus (L \setminus L_\beta)^2).\]
\end{proof}

As mentioned in the introduction, another type of operator that can arise on a space of the form $C_0(L \times L)$ is one obtained from the tensor products of operators on $C_0(L)$. For example, given operators $T_1, T_2: C_0(L) \to C_0(L)$, and noting that $C_0(L) \hat{\otimes}_{\varepsilon} C_0(L)$ is linearly isometric to $C_0(L \times L)$ (see \cite[Theorem 20.5.6]{Se}), there is an unique operator $T_1 \otimes T_2: C_0(L \times L) \to C_0(L \times L)$ that satisfies $T_1 \otimes T_2(f \otimes g)(x, y) = T_1(f)(x) T_2(g)(y)$, where $f \otimes g(x,y) = f(x)g(y)$, with norm $\|T_1 \otimes T_2\| = \|T_1\| \|T_2\|$ (see \cite[Theorem 20.5.3]{Se}). According to Theorem \ref{ReducaoOperador}, for every operator $T: C_0(L) \to C_0(L)$, there exist a scalar $p \in \mathbb{R}$ and an operator $R$ with separable range such that $T = pI + R$. Hence, for operators $T_1, T_2, T_3, T_4: C_0(L) \to C_0(L)$, if $T$ denotes the operator $T_1 \otimes T_4 + J \circ (T_2 \otimes T_3)$, then $T$ can be represented as 
\[
T = pI + qJ + R_1 \otimes I + J \circ (R_2 \otimes I) + J \circ (I \otimes R_3) + I \otimes R_4 + S,
\]
where $I$ denotes the identity operator on $C_0(L \times L)$, $S$ is an operator on $C_0(L \times L)$ with separable range, and $R_i$ ($i=1,2,3,4$) denote operators on $C_0(L)$ with separable range. Since the term $pI + qJ$ will be absorbed in the initial step of our decomposition process, the critical part of the operator above is what we refer to as the matrix operator generated by the operators $R_1, R_2, R_3, R_4: C_0(L) \to C_0(L)$. In the third stage of our decomposition, we will isolate a component of the operator that corresponds to a matrix operator.

It is important to note that each component of a matrix operator can be understood as a lifting of a separable range operator $R: C_0(L) \to C_0(L)$ to an operator on $C_0(L \times L)$. It is easily seen that these components satisfy the following formulas:
 
\begin{align*}
R_1^{(1)}(f)(x,y) &= (R_1 \otimes I)(f)(x,y) = R_1(f^y)(x), \\
R_2^{(2)}(f)(x,y) &= J \circ (R_2 \otimes I)(f)(x,y) = R_2(f^x)(y), \\
R_3^{(3)}(f)(x,y) &= J \circ (I \otimes R_3)(f)(x,y) = R_3(f_y)(x), \\
R_4^{(4)}(f)(x,y) &= (I \otimes R_4)(f)(x,y) = R_4(f_x)(y),
\end{align*}
where, for every $f \in C_0(L \times L)$, $f^y$ denotes the function in $C_0(L)$ given by $f^y(x) = f(x,y)$, and $f_x$ denotes the function in $C_0(L)$ given by $f_x(y) = f(x,y)$.

\begin{definition}\label{Def:Matrizes}
Let $R_1, R_2, R_3, R_4 : C_0(L) \to C_0(L)$ be separable range operators. The matrix operator generated by $R_1, R_2, R_3, R_4$ is the mapping 
\[\left( \begin{smallmatrix}
 R_1 & R_3 \\
 R_2 & R_4
\end{smallmatrix} \right) : C_0(L \times L) \to C_0(L \times L)\]
given by the formula
\[\left( \begin{smallmatrix}
 R_1 & R_3 \\
 R_2 & R_4
\end{smallmatrix} \right)(f)(x,y) = R_1^{(1)}(f)(x,y) + R_2^{(2)}(f)(x,y) + R_3^{(3)}(f)(x,y) + R_4^{(4)}(f)(x,y).\]
\end{definition}

The following proposition establishes an important fact about matrix operators:

\begin{proposition}\label{Prop:BlockZero}
Let $R_1, R_2, R_3, R_4 : C_0(L) \to C_0(L)$ be separable range operators and let $R=\left( \begin{smallmatrix}
 R_1 & R_3 \\
 R_2 & R_4
\end{smallmatrix} \right)$ be the matrix operator generated by these mappings. There exists $\rho < \omega_1$ such that 
\begin{itemize}
\item $R[C_0((L \setminus L_\rho)^2)] = \{\vec{0}\}$
\item $R[C_0(L^2 \setminus (L \setminus L_\rho)^2)] \subseteq C_0(L^2 \setminus (L \setminus L_\rho)^2)$
\end{itemize}
\end{proposition}
\begin{proof}
Since $R_1, R_2, R_3, R_4: C_0(L) \to C_0(L)$ are separable range operators, there exists some $\alpha < \omega_1$ such that $R_i[C_0(L)] \subset C_0(L_\alpha)$ for each $i = 1, 2, 3, 4$. It follows from Definition \ref{Def:Matrizes} that for every $f \in C_0(L \times L)$ and $(x, y) \in (L \setminus L_\alpha)^2$, we have $R(f)(x, y) = 0$.

Next, observe that we can choose $\beta$ such that $R_i[C_0(L \setminus L_\beta)] = \{\vec{0}\}$ for every $i = 1, 2, 3, 4$. If this were not the case, for each $\beta < \omega_1$, there exists $f_\beta \in C_0(L \setminus L_\beta)$ such that $R_i(f_\beta) \neq 0$ for some $i = 1, 2, 3, 4$. This implies the existence of some $x \in L_\alpha$ and an uncountable family $\{y_\beta : \beta < \omega_1\}$ such that for some $i = 1, 2, 3, 4$, $R_i^*(\delta_x)(\{y_\beta\}) \neq 0$ for uncountably many $\beta < \omega_1$. This is a contradiction, since $R_i^*(\delta_x)$ has countable support.

Therefore, by Definition \ref{Def:Matrizes}, we can verify that $R(g) = 0$ for every $g \in C_0((L \setminus L_\beta)^2)$. We conclude the proof by taking $\rho = \max\{\alpha, \beta\}$.
\end{proof}

\begin{remark}\label{Rem:CountableDomain}
As a consequence from the previous proposition, it follows that the diagonal operators $M_g$ and $N_h$ can be represented as matrix operators only in the trivial case where $g$ and $h$ are zero functions. Indeed, if $R$ denotes a matrix operator, then there exists $\rho < \omega_1$ such that $R[C_0((L \setminus L_\rho)^2)] = \{\vec{0}\}$. If $R = N_h$ for some function $h \in C_0(L)$, then by fixing an isolated point $y_0 \in L \setminus L_\rho$ and letting $f = \chi_{\{(y_0, y_0)\}}$, we find that for every $x \in L$, $R(f)(x, y_0) = 0$, while $N_h(x, y_0) = h(x)$. This implies that $h = 0$. Using a similar argument, we obtain the same result for $M_g$.
\end{remark}

\begin{proposition}\label{oper8}
Let $R_1, R_2, R_3, R_4 : C_0(L) \to C_0(L)$ be separable range operators. The matrix operator 
$\left( \begin{smallmatrix}
 R_1 & R_3 \\
 R_2 & R_4
\end{smallmatrix} \right)$ has a separable image if and only if $R_1 = R_2 = R_3 = R_4 = 0$.
\end{proposition}

\begin{proof}
Assume that $R = \left( \begin{smallmatrix}
 R_1 & R_3 \\
 R_2 & R_4
\end{smallmatrix} \right)$ has a separable image. By arguing as in \cite[Proposition 3.5]{candido2}, we can obtain $\rho < \omega_1$ such that $R[C_0(L \times L)] \subset C_0(L_\rho \times L_\rho)$. Moreover, as discussed in the proof of Proposition \ref{Prop:BlockZero}, we may assume that $R_i[C_0(L)] \subset C_0(L_\rho)$ and $R_i[C_0(L \setminus L_\rho)] = \{\vec{0}\}$ for every $i = 1, 2, 3, 4$. 

Now, pick an isolated point $y_0 \in L \setminus L_\rho$. For an arbitrary function $h \in C_0(L)$, consider $H = h \otimes \chi_{\{y_0\}} \in C_0(L \times L)$. For each $x \in L_\rho$, we have
\[
R_1(h)(x) = R_1((h \otimes \chi_{\{y_0\}})^{y_0})(x) = \left( \begin{smallmatrix}
 R_1 & R_2 \\
 R_3 & R_4
\end{smallmatrix} \right)(H)(x, y_0) = 0.
\]
By similar reasoning, we obtain $R_2(h)(x) = R_3(h)(x) = R_4(h)(x) = 0$. Thus, we deduce that $R_1 = R_2 = R_3 = R_4 = 0$.
\end{proof}

The following essential lemma is inspired by \cite[Theorem 4.8]{candido2}.

\begin{lemma}\label{Lem:AuxiliarDecomposition}
Let $S:C_0(L\times L) \to C_0(L\times L)$ be an operator. For every $1\leq \alpha < \omega_1$, there exist $\alpha\leq \rho<\omega_1$ and functions $r_1, r_2, r_3, r_4 : L_\rho \times L_\rho \to \mathbb{R}$ such that
\[
\begin{array}{cc}
S^*(\delta_{(a,x)})(\{(b,y)\}) = \left\{
\begin{array}{ll}
r_1(a,b) & \text{if } x = y; \\
0 & \text{if } x \neq y
\end{array} \right.
& 
S^*(\delta_{(a,x)})(\{(y,b)\}) = \left\{
\begin{array}{ll}
r_3(a,b) & \text{if } x = y; \\
0 & \text{if } x \neq y
\end{array} \right.
\\[1.5em]
S^*(\delta_{(x,a)})(\{(b,y)\}) = \left\{
\begin{array}{ll}
r_2(a,b) & \text{if } x = y; \\
0 & \text{if } x \neq y
\end{array} \right.
& 
S^*(\delta_{(x,a)})(\{(y,b)\}) = \left\{
\begin{array}{ll}
r_4(a,b) & \text{if } x = y; \\
0 & \text{if } x \neq y
\end{array} \right.
\end{array}
\]
for all $a, b \in L_\rho$ and $x, y \in L \setminus L_\rho$.
\end{lemma}
\begin{proof}
Let $1 \leq \alpha < \omega_1$ be a fixed ordinal, and let $a, b \in L$ be arbitrary. We can choose a sequence of clopen compact neighborhoods $V_1 \supset V_2 \supset V_3 \supset \cdots$ of $b$ such that $\bigcap_{n < \omega} V_n = \{b\}$. For each $n < \omega$, we define the functions $\nu^{(1,n)}, \nu^{(2,n)}, \nu^{(3,n)}, \nu^{(4,n)}: K \to M_0(K)$ by:
\[\begin{array}{cc}
\nu^{(1,n)}_{x}(U) = S^*(\delta_{(a,x)})(V_n \times U), & \nu^{(3,n)}_{x}(U) = S^*(\delta_{(a,x)})(U \times V_n) , \\[1.5em]
\nu^{(2,n)}_{x}(U) = S^*(\delta_{(x,a)})(V_n \times U), & \nu^{(4,n)}_{x}(U) = S^*(\delta_{(x,a)})(U \times V_n).
\end{array}\]
It is readily seen that each of the functions above is weak$^*$ continuous and vanishes at infinity. We will work with the sequence $\{\nu^{(1,n)}\}_n$; similar arguments apply to the other sequences. For each $n < \omega$, from Theorems \ref{Thm:vanish-B} and \ref{Thm:vanish-A}, there exist $s_n \in \mathbb{R}$ and $\beta_n < \omega_1$ such that
\[\nu^{(1,n)}_{x}(\{y\})= \begin{cases}
s_n, & \text{if } x = y, \\
0, & \text{if } x \neq y,
\end{cases}\]
for all $x,y \in L \setminus L_{\beta_n}$. We define $\xi_1(a,b) = \sup_{n < \omega} \beta_n$. For every $m, n < \omega$ and $x \in L \setminus L_{\xi_1(a,b)}$, assuming that $n \leq m$, we have
\[
|s_n - s_m| = |\nu^{(1,n)}_{x}(\{x\}) - \nu^{(1,m)}_{x}(\{x\})| = |S^*(\delta_{(a,x)})((V_n \setminus V_m) \times \{x\})| \leq |S^*(\delta_{(a,x)})((V_n \setminus \{b\}) \times \{x\})|.
\]
Since $\bigcap_{n<\omega} (V_n \setminus \{b\}) \times \{x\} = \emptyset$, it follows that
\[
\lim_{n \to \infty} |S^*(\delta_{(a,x)})((V_n \setminus \{b\}) \times \{x\})| = 0,
\]
and we may deduce that the sequence $\{s_n\}_n$ converges to a limit $r_1(a,b) \in \mathbb{R}$. Hence, for each $x \in L \setminus L_{\xi_1(a,b)}$ and $n < \omega$, we have
\[
|S^*(\delta_{(a,x)})(\{(b,x)\}) - r_1(a,b)| \leq |S^*(\delta_{(a,x)})((V_n \setminus \{b\}) \times \{x\})| + |s_n - r_1(a,b)|,
\]
from which it follows that $S^*(\delta_{(a,x)})(\{(b,x)\}) = r_1(a,b)$. Furthermore, for all $x, y \in L \setminus L_{\xi_1(a,b)}$ and $n < \omega$, if $x \neq y$, we obtain
\[
|S^*(\delta_{(a,x)})(\{(b,y)\})| \leq |S^*(\delta_{(a,x)})(V_n \times \{y\})| + |S^*(\delta_{(a,x)})((V_n \setminus \{b\}) \times \{y\})|,
\]
and since $\bigcap_{n<\omega} (V_n \setminus \{b\}) \times \{y\} = \emptyset$, it follows that $S^*(\delta_{(a,y)})(\{(b,x)\}) = 0$. 

By applying similar arguments to the other sequences $\{\nu^{(2,n)}\}_n$, $\{\nu^{(3,n)}\}_n$, and $\{\nu^{(4,n)}\}_n$, we obtain, for each $(a,b) \in L \times L$, ordinals $\xi_2(a,b)$, $\xi_3(a,b)$, and $\xi_4(a,b)$, and real numbers $r_2(a,b)$, $r_3(a,b)$, and $r_4(a,b)$, respectively. 

We set $\rho_0 = \alpha$ and, by letting $\xi(a,b) = \max\{\xi_j(a,b) : j = 1,2,3,4\}$, we define
\[
\rho_{n+1} = \sup\{\xi(a,b) : (a,b) \in L_{\rho_n} \times L_{\rho_n}\},
\]
for each $n < \omega$. Finally, we take $\rho = \sup_{n < \omega} \rho_n$, and we complete the proof by defining, for each $i = 1, 2, 3, 4$, the functions $r_i: L_\rho \times L_\rho \to \mathbb{R}$ by the formula $(a,b) \mapsto r_i(a,b)$.

\end{proof}

We will now proceed with the third and final step of the decomposition, which will ultimately lead to the proof of our main result.

\begin{proof}[Proof of Theorem \ref{main1}]
Let $T:C_0(L\times L)\to C_0(L\times L)$ be an operator. According to Theorem \ref{ReducaoOperadorL2}, there are scalars $p$ and $q$, an ordinal $\beta < \omega_1$, and functions $g, h \in C_0(L_\beta)$ such that, if we define $Q = T - p I - q J - M_g - N_h$, then:
\begin{itemize}
\item $Q[C_0((L\setminus L_\beta)^2)]=\{\vec{0}\}$
\item $Q[C_0(L^2\setminus (L\setminus L_\beta)^2)]\subseteq C_0(L^2\setminus (L\setminus L_\beta)^2)$
\end{itemize}
Let $\beta \leq \rho < \omega_1$, and let $r_1, r_2, r_3, r_4 : L_\rho \times L_\rho \to \mathbb{R}$ be the functions associated with the operator $Q$, as described in Lemma \ref{Lem:AuxiliarDecomposition}. For each $i = 1, 2, 3, 4$, we extend $r_i$ to $K \times K$ by setting $r_i(a, b) = 0$ whenever $(a, b) \in K^2 \setminus L_\rho^2$, and define the map $R_i : C_0(L) \to C_0(L)$ via the formula:
\[
R_i(f)(a) = \sum_{b \in L} r_i(a, b)  f(b).
\]
To verify that $R_1$ is a well-defined bounded linear mapping, consider an isolated point $x_0 \in L \setminus L_\rho$. Notice that for each $f \in C_0(L_\rho)$ and $a \in L_\rho$, we have:
\begin{align*}
Q(f \otimes \chi_{\{x_0\}})(a, x_0) &= Q^*(\delta_{(a, x_0)})(f \otimes \chi_{\{x_0\}}) \\
&= \sum_{b \in L_\rho} f(b) Q^*(\delta_{(a, x_0)})(\{(b, x_0)\})= \sum_{b \in L_\rho} r_1(a, b) f(b) = R_1(f)(a).
\end{align*}
For each $f \in C_0(L \setminus L_\rho)$, we have $R_1(f) = 0$. Since $Q$ is a bounded linear mapping, so is $R_1$. Moreover, $R_1$ is a separable range operator, as $R_1[C_0(L)] \subset C_0(L_\rho)$. The same conclusion holds for the mappings $R_2, R_3$, and $R_4 : C_0(L) \to C_0(L)$.

Now, consider the matrix operator $R=\begin{pmatrix} R_1 & R_3 \\ R_2 & R_4 \end{pmatrix}$. Define $S = Q -R$ and, recalling Proposition \ref{Prop:BlockZero}, we notice that 
\begin{itemize}
\item $S[C_0((L\setminus L_\rho)^2)]=\{\vec{0}\}$
\item $S[C_0(L^2\setminus (L\setminus L_\rho)^2)]\subseteq C_0(L^2\setminus (L\setminus L_\rho)^2)$
\end{itemize}
Toward a contradiction, suppose that $S$ does not have a separable image. From the relations above, we can conclude that for each $\alpha < \omega_1$, there exists $f_\alpha \in C_0(L^2 \setminus (L \setminus L_\rho)^2)$ such that $S(f_\alpha) \notin C_0(L_\alpha \times L_\alpha)$. Consequently, there must exist $(a_\alpha, x_\alpha) \in (L^2 \setminus (L \setminus L_\rho)^2) \setminus L_\alpha^2$ such that $S(f_\alpha)(a_\alpha, x_\alpha) \neq 0$.

Now, we can assume that there is an uncountable subset $\varGamma \subset \omega_1$ such that for each $\alpha \in \varGamma$, there exists $(a_\alpha, x_\alpha) \in L_\rho \times (L \setminus L_\alpha)$ with $S(f_\alpha)(a_\alpha, x_\alpha) \neq 0$. The complementary case can be handled with a similar argument.

By passing to an uncountable subset, and since $L_\rho$ is countable, we may assume that $a_\alpha = a$ for all $\alpha \in \varGamma$, and the points $x_\alpha$ are all distinct elements of $L\setminus L_\rho$. For every $\alpha \in \varGamma$, we have
\[S(f_{\alpha})(a, x_{\alpha})= \int f_{\alpha} \, dS^*(\delta_{(a, x_{\alpha})})= \sum_{(z, y) \in L^2 \setminus (L \setminus L_\rho)^2} f_{\alpha}(z, y) S^*(\delta_{(a, x_{\alpha})})(\{(z, y)\})\]
and we may fix $(z_{\alpha}, y_{\alpha}) \in L^2 \setminus (L \setminus L_\rho)^2$ such that $S^*(\delta_{(a, x_{\alpha})})(\{(z_{\alpha}, y_{\alpha})\}) \neq 0$. By passing to a further uncountable subset, we may assume that $(z_{\alpha}, y_{\alpha}) \in L_\rho \times L$ for each $\alpha \in \varGamma$ (the other case can be proved with a similar argument). Furthermore, since $L_\rho$ is countable, we may assume that $z_\alpha = b$ for each $\alpha\in \varGamma$. Additionally, recalling Remark \ref{Rem:SeAnulaNoInfinito}, we assume that the points $y_\alpha$ are all distinct elements of $L\setminus L_\rho$.

If there exists some $\alpha \in \varGamma$ such that $x_\alpha \neq y_\alpha$, we obtain
\[S^*(\delta_{(a, x_{\alpha})})(\{(b, y_{\alpha})\}) = Q^*(\delta_{(a, x_{\alpha})})(\{(b, y_{\alpha})\}) - R^*(\delta_{(a, x_{\alpha})})(\{(b, y_{\alpha})\}) = 0,\]
which is a contradiction. Therefore, we can conclude that $x_\alpha = y_\alpha$ for all $\alpha$. However, in this case, we have
\begin{align*}
    S^*(\delta_{(a, x_{\alpha})})(\{(b, x_{\alpha})\}) &= Q^*(\delta_{(a, x_{\alpha})})(\{(b, x_{\alpha})\}) - (R_1^{(1)})^*(\delta_{(a, x_{\alpha})})(\{(b, x_{\alpha})\}) \\
    &= r_1(a, b) - r_1(a, b) = 0,
\end{align*}
which is again a contradiction. We deduce that there exists some $\alpha < \omega_1$ such that $S[C_0(L \times L)] \subset C_0(L_\alpha \times L_\alpha)$, implying that $S$ has a separable image.

To establish the uniqueness of the decomposition, assume that there exist $p', q' \in \mathbb{R}$, functions $g', h' \in C_0(L)$, separable range operators $R'_1, R'_2, R'_3, R'_4 : C_0(L) \to C_0(L)$, and an operator with a separable image $S' : C_0(L \times L) \to C_0(L \times L)$ such that, if we define $R' = \begin{pmatrix} R'_1 & R'_3 \\ R'_2 & R'_4 \end{pmatrix}$, then
\[
T = p' I + q'J + M_{g'} + N_{h'} + R' + S'.
\]

Recalling Proposition \ref{Prop:BlockZero}, and since $f'$ and $h'$ have countable support, there exists $\rho \leq \xi < \omega_1$ such that 
$f', g' \in C_0(L_\xi)$, $R'[C_0((L \setminus L_\xi)^2)] = \{\vec{0}\}$, and $S'[C_0(L \times L)] \subset C_0(L_\xi \times L_\xi)$. Fix distinct isolated points $x_0, y_0$ in $L \setminus L_\xi$. By computing $T(\chi_{(x_0, y_0)})(x_0, y_0)$ and $T(\chi_{(x_0, y_0)})(y_0, x_0)$, we deduce that $p = p'$ and $q = q'$. 

Next, for each $x, y \in L_\xi$, by evaluating $T(\chi_{(x_0, x_0)})(x_0, y)$ and $T(\chi_{(y_0, y_0)})(x, y_0)$, we deduce that $g(y) = g'(y)$ and $h(x) = h'(x)$. Thus, we have
\[
\begin{pmatrix} 
R_1 - R'_1 & R_3 - R'_3 \\ 
R_2 - R'_2 & R_4 - R'_4 
\end{pmatrix} = S' - S,
\]
and consequently, the matrix operator on the left-hand side has a separable image. By Proposition \ref{oper8}, we conclude that $R_i = R'_i$ for all $i = 1, 2, 3, 4$, and therefore $S' = S$.

\end{proof}

\section{The geometry of $C_0(L\times L)$}
\label{Sec-Geo}

In this section, we will discuss some aspects of the geometry of the spaces \(C_0(L \times L)\), where \(L\) is the locally compact space constructed in Section \ref{Sec-Ostaszewski}. The first result in this direction is as follows:

\begin{theorem}
$C_0(L \times L)$ is not isomorphic to its square.
\end{theorem}

\begin{proof}
Towards a contradiction, assume that there are operators $T_1, T_2: C_0(L \times L) \to C_0(L \times L)$ such that the formula $T(f) = (T_1(f), T_2(f))$ defines a surjective map of $C_0(L \times L)$ onto $C_0(L \times L) \oplus C_0(L \times L)$.

From Theorem \ref{ReducaoOperador}, there exist scalars $p_1, q_1, p_2, q_2$ and $\rho < \omega_1$ such that, if $S_1 = T_1 - p_1  I - q_1  J$ and $S_2 = T_2 -p_2  I - q_2 J$, we have 
\[S_1[C_0(L\times L)]\cup S_2[C_0(L\times L)]\subset C_0(L^2\setminus (L\setminus L_\rho)^2).\]

Let $x_0, y_0$ be distinct isolated points of $L \setminus L_\rho$, and let $f \in C_0(L \times L)$ be such that $T(f) = (\chi_{(x_0,x_0)}, \chi_{(y_0,y_0)})$. It follows that
\begin{align*}
    p_1 f(x,y) + q_1 f(y,x) + S_1(f)(x,y) &= \chi_{(x_0,x_0)}(x,y), \\
    p_2 f(x,y) + q_2 f(y,x) + S_2(f)(x,y) &= \chi_{(y_0,y_0)}(x,y).
\end{align*}

Since $S_i(f)(x_0,x_0)=S_i(f)(y_0,y_0)=0$, $i=1,2$, we have:
\[
\begin{array}{cc}
    \begin{array}{rcl}
        (p_1 + q_1) f(x_0,x_0) &=& 1, \\
        (p_2 + q_2) f(x_0,x_0) &=& 0,
    \end{array}
    &
    \quad
    \begin{array}{rcl}
        (p_1 + q_1) f(y_0,y_0) &=& 0, \\
        (p_2 + q_2) f(y_0,y_0) &=& 1.
    \end{array}
\end{array}
\]

This leads to a contradiction.
\end{proof}

As we have seen, for an arbitrary locally compact space $L$, the most elementary operators on $C_0(L^n)$ are inevitably those induced by the permutation of coordinates. These operators will ultimately induce projections of the form $P = a_1 I_{\sigma_1} + \ldots + a_r I_{\sigma_r}$, which we refer to as \emph{natural projections}. In general, one can find all the natural projections induced by $\mathfrak{S}_n$ by finding all solutions to the equation:
\[
\left( \sum_{\sigma \in \mathfrak{S}_n} x_\sigma I_\sigma \right)^2 = \sum_{\sigma \in \mathfrak{S}_n} x_\sigma I_\sigma.
\]
In our particular case, for $C_0(L \times L)$, solving the above equation for the case $n=2$ and returning to the notation of the previous section, where $I$ represents the identity in $C_0(L \times L)$ and $J$ the function that transposes the coordinates in the function, besides the trivial projections, which are the null operator and the identity operator, we obtain:
\[ 
P_1 = \frac{1}{2}(I + J) \quad \text{and} \quad P_2 = \frac{1}{2}(I - J).
\]
In other words, $P_1$ is the projection onto the subspace of $C_0(L \times L)$ consisting of all symmetric functions, i.e., functions $f \in C_0(L \times L)$ that satisfy $f(x, y) = f(y, x)$, and $P_2$ projects onto the space of anti-symmetric functions that satisfy $f(x, y) = -f(y, x)$.

It is straightforward to verify that there exist locally compact Hausdorff spaces $\varOmega_1$ and $\varOmega_2$ such that $P_1[C_0(L \times L)] \cong C_0(\varOmega_1)$ and $P_2[C_0(L \times L)] \cong C_0(\varOmega_2)$. As we will see next, in our specific case, these spaces are not isomorphic, as demonstrated by the following result.

\begin{proposition}\label{Prop:NotIsomorphic}
There is no injective operator from $C_0(\varOmega_1)$ into $C_0(\varOmega_2)$.
\end{proposition}
\begin{proof}
If $T:C_0(\varOmega_1)\to C_0(\varOmega_2)$ denotes a bounded linear mapping, let $Q:C_0(L\times L)\to C_0(L\times L)$ be defined by  $Q=T\circ P_1$, where $P_1=\frac{1}{2}(I+J)$. According to Theorem \ref{ReducaoOperador}, there are $\rho<\omega_1$,  scalars $p$ and $q$ and functions $g,h \in C_0(L_\rho)$ such that, if $S=Q-p I-q J-M_g-N_h$, then  
\[S[C_0(L^2)]\subset C_0(L^2\setminus (L\setminus L_\rho)^2) \quad \text{ and } \quad S[C_0((L\setminus L_\rho)^2)]=\{\vec{0}\}.\]

Since $C_0(\varOmega_2)$ is isometrically identified with the space of anti-symmetric function of $C_0(L\times L)$, for every function
$f\in C_0(L\times L)$ we have
\[\frac{1}{2}(Q-J\circ Q)(f)(x,y)=Q(f)(x,y)\]
for all $(x,y)$. In particular, if $(x,y)\in (L\setminus L_\rho)^2$, then $S(f)(x,y)=0$ and we obtain
\[\frac{1}{2}(p-q) f(x,y)+\frac{1}{2}(q-p) f(y,x)=p f(x,y)+q f(y,x).\]

Let $x_0$ and $y_0$ be distinct isolated point of $L\setminus L_\rho$ and let $f=\frac{1}{2}(\chi_{(x_0,y_0)}+\chi_{(y_0,x_0)})$. The previous relation applied to $f$ leads to $p+q=0$. Moreover, since $f\in C_0(\varOmega_1)$ and clearly $M_g(f)=N_f(f)=0$ we have
\[T(f)(x,y)=(T\circ P_1)(f)(x,y)=Q(f)(x,y)=p f(x,y)+qf(y,x)=0\]
for all $(x,y)$. Therefore, $T$ is not injective.
\end{proof}

With minor adjustments to the previous proposition, it is also possible to demonstrate that there is no injective bounded linear mapping from $C_0(\varOmega_2)$ to $C_0(\varOmega_1)$. We deduce that $C_0(\varOmega_1)$ is not isomorphic to $C_0(\varOmega_2)$, and neither of these spaces is isomorphic to $C_0(L \times L)$. 

We now observe that the locally compact space $L$, constructed in Section \ref{Sec-Ostaszewski}, satisfies the property of being a $3$-collapsing space of height $\omega$, as defined in \cite[Definition 4.1]{candido2}. Consequently, all the results from \cite{candido2} apply to our space $L$, leading to the following isomorphisms:

\begin{proposition}\label{Prop:Simplification}
For every $\rho < \omega_1$, if $L_\rho$ is infinite, then either $C_0(L_\rho) \sim C_0(\omega)$ or $C_0(L_\rho) \sim C_0(\omega^\omega)$. Hence, if $L_\rho$ is infinite, both $C_0(L_\rho \times L)$ and $C_0(L^2 \setminus (L \setminus L_\rho)^2)$ are linearly isomorphic to either $C_0(\omega \times L)$ or $C_0(\omega^\omega \times L)$.
\end{proposition}
\begin{proof}
The first statement follows from \cite[Proposition 5.4]{candido2}. From \cite[Proposition 5.5]{candido2}, we have $C_0(L \setminus L_\rho) \sim C_0(L)$. Then, recalling \cite[Theorem 20.5.6]{Se}, assuming that $C_0(L_\rho) \sim C_0(\omega^\omega)$, we have
\begin{align*}
    C_0(L_\rho \times L) &\sim C_0(L_\rho) \hat{\otimes}_{\varepsilon} C_0(L) \sim C_0(\omega^\omega) \hat{\otimes}_{\varepsilon} C_0(L) \sim C_0(\omega^\omega \times L).
\end{align*}
Consequently,
\begin{align*}
    C_0(L^2 \setminus (L \setminus L_\rho)^2) &\sim C_0(L_\rho \times L) \oplus C_0((L \setminus L_\rho) \times L_\rho) \\
    &\sim C_0(L_\rho \times L) \oplus C_0(L_\rho \times L) \\
    &\sim C_0(\omega^\omega \times L) \oplus C_0(\omega^\omega \times L) \\
    &\sim C_0((\omega^\omega \oplus \omega^\omega) \times L) \sim C_0(\omega^\omega \times L).
\end{align*}
The remaining case can be proved similarly.
\end{proof}

We close this research by proving that neither $C_0(L \times L)$ nor $C_0(\varOmega_1)$ or $C_0(\varOmega_2)$ are isomorphic to any subspace of $C_0(\omega^\omega \times L)$. This will be a consequence of the following proposition:

\begin{proposition}\label{Prop-ClassificacaoFinal}
Let $X$ denote either $C_0(L \times L)$ or $C_0(\Omega_i)$, where $i = 1, 2$. Then, there is no injective operator from $X$ into $C_0(L_\alpha \times L)$ for any $\alpha < \omega_1$.
\end{proposition}
\begin{proof}
Let $\alpha < \omega_1$ be arbitrary and consider a bounded linear mapping $T: C_0(\varOmega_1) \to C_0(L_\alpha \times L)$. If $P_1 = \frac{1}{2}(I + J)$, let $Q: C_0(L \times L) \to C_0(L \times L)$ be defined as $Q = T \circ P_1$. According to Theorem \ref{ReducaoOperadorL2}, there exists $\rho < \omega_1$, scalars $p$ and $q$, and functions $g, h \in C_0(L_\rho)$ such that $S = Q - p I - q J - M_g - N_h$ satisfies $S[C_0((L \setminus L_\rho)^2)] = \{\vec{0}\}$.

Let $x_0$ and $y_0$ be distinct isolated points in $L \setminus (L_\alpha \cup L_\rho)$ and fix $f = \frac{1}{2}(\chi_{(x_0, y_0)} + \chi_{(y_0, x_0)})$. It is readily seen that $S(f) = M_g(f) = N_h(f) = 0$. Moreover, since $Q[C_0(L \times L)] \subset C_0(L_\alpha \times L)$, we obtain $Q(f)(x_0, y_0) = \frac{1}{2}(p + q) = 0$. 

Now, since $f \in C_0(\varOmega_1)$, we have for all $(x, y) \in L \times L$,
\begin{align*}
T(f)(x, y) = T \circ P_1(f)(x, y) =Q(f)(x, y) = \frac{p + q}{2} (\chi_{(x_0, y_0)} + \chi_{(y_0, x_0)})(x, y) = 0,
\end{align*}
which shows that $T$ is not injective. With similar arguments, one may check that there is also no bounded injective linear mapping from $C_0(\varOmega_2)$ into $C_0(L_\alpha \times L)$.
\end{proof}

\section{Acknowledgements}

This work was supported by Fundação de Amparo \`a Pesquisa do Estado de S\~ao Paulo (FAPESP), grant no. 2023/12916-1.

\bibliographystyle{amsalpha}

\end{document}